\documentclass[12pt]{amsart}
\usepackage{amsmath}
\usepackage{amsxtra}
\usepackage{amscd}
\usepackage{amsthm}
\usepackage{amsfonts}
\usepackage{amssymb}
\usepackage {pstricks}
\usepackage {pmat}
\usepackage{pstricks,pst-node}
\usepackage[all]{xy}

\usepackage{verbatim}

\newtheorem{theorem}{Theorem}[section]

\newtheorem{lemma}[theorem]{Lemma}

\newtheorem{corollary}[theorem]{Corollary}
\newtheorem{proposition}[theorem]{Proposition}

\theoremstyle{definition}
\newtheorem{definition}[theorem]{Definition}
\newtheorem{example}[theorem]{Example}
\newtheorem{scholium}[theorem]{Scholium}

\theoremstyle{remark}
\newtheorem{remark}[theorem]{Remark}

\numberwithin{equation}{section}

\newcommand{\be}{\begin{equation}}
\newcommand{\ee}{\end{equation}}

\newcommand{\mc}{\mathcal}

\newcommand{\C}{{\mathbb C}}

\newcommand{\Z}{{\mathbb Z}}
\newcommand{\N}{{\mathbb N}}

\newcommand{\bE}{{\mathbb E}}
\newcommand{\bS}{{\mathbb S}}

\newcommand{\id}{{\rm{id}}}

\newcommand{\mf}{\mathfrak}
\newcommand{\fg}{{\mf g}}

\newcommand{\fb}{{\mf b}}

\newcommand{\cB}{\mc B}

\newcommand{\cE}{\mc E}
\newcommand{\cF}{\mc F}
\newcommand{\cH}{\mc H}
\newcommand{\cI}{\mc I}

\newcommand{\cR}{\mc R}
\newcommand{\cS}{\mc S}
\newcommand{\cT}{\mc T}

\newcommand{\cP}{\mc P}
\newcommand{\bv}{\mathbf v}

\newcommand{\La}{\Lambda}
\newcommand{\la}{\lambda}
\newcommand{\lpq}{\Lambda^{p|q}}

\newcommand{\cM}{\mc M}

\newcommand{\U}{{\rm{U}}}
\newcommand{\End}{{\rm{End}}}

\newcommand{\Hom}{{\rm{Hom}}}

\newcommand{\GL}{{\rm{GL}}}
\newcommand{\Sym}{{\rm{Sym}}}
\newcommand{\Exp}{{\rm{Exp}}}

\newcommand{\im}{{\rm{Im}}}
\newcommand{\inv}{{^{-1}}}
\newcommand{\st}{^{\rm{st}}}

\advance\headheight by 2pt

\newcommand{\fso}{{\mathfrak {so}}}

\newcommand{\gl}{{\mathfrak {gl}}}

\newcommand{\osp}{{\mathfrak {osp}}}

\newcommand{\lr}{\longrightarrow}
\newcommand{\ol}{\overline}

\newcommand{\ot}{\otimes}

\newcommand{\sdim}{{\rm sdim\,}}

\newcommand{\OSp}{{\rm OSp}}

\begin{document}

\normalfont

\title[Invariant theory of the orthosymplectic supergroup]{The first fundamental theorem of
invariant\\ theory for the  orthosymplectic supergroup}

\author{G.I. Lehrer and R.B. Zhang}
\thanks{This research was supported by the Australian Research Council}
\address{School of Mathematics and Statistics,
University of Sydney, N.S.W. 2006, Australia}
\email{gustav.lehrer@sydney.edu.au, ruibin.zhang@sydney.edu.au}
\date {Final draft}
\begin{abstract}
We give an elementary proof of the first fundamental theorem of invariant theory for the
orthosymplectic supergroup by generalising the geometric method of Atiyah, Bott and Patodi to
the supergroup context.  We use methods from super-algebraic geometry to convert invariants of
the orthosymplectic supergroup into invariants of the corresponding general linear supergroup on a different
space. In this way, super Schur-Weyl-Brauer duality is established
between the orthosymplectic supergroup of superdimension
$(m|2n)$ and the Brauer algebra with parameter $m-2n$. The result may be interpreted
in terms of the relevant Harish-Chandra super pair action (over $\C$), or
equivalently, in terms of  the orthosymplectic Lie supergroup over the infinite dimensional
Grassmann algebra. We also state a corresponding theorem for the orthosymplectic Lie superalgebra,
which involves an extra invariant generator, the super-Pfaffian.
\end{abstract}



\maketitle

\tableofcontents

\section{Introduction}\label{sect:intro}

The invariant theory of Lie supergroups and Lie superalgebras plays an important
role in the study of supersymmetry in particle physics. For example,
the fields in any supersymmetric quantum field theory
admit a nontrivial action of the Poincar\'e supergroup. The hamiltonian of the theory,
which is constructed from the quantum fields, is invariant under the Poincar\'e supergroup action,
and so also are the other physical quantities which may be measured experimentally.

Systematic studies of the invariant theory of Lie supergroups and Lie superalgebras \cite{K,S}
already appear in the early 80s in works \cite{Sch1, Sch2, Sch3}  of Scheunert, who
in particular constructed generators for the centres of the universal enveloping
algebras of the classical Lie superalgebras.
Sergeev announced fundamental theorems of invariant theory
for classical supergroups and their Lie superalgebra in \cite{S0}
and provided more detailed expositions in \cite{S1, S2}.
A super analogue of the Schur-Weyl duality between the general linear superalgebra
and the symmetric group was established  independently in \cite{BR, S0}.

More recently,  Brundan and Stroppel in \cite{BS} showed that there is a surjection
from the walled Brauer algebra to the endomorphism algebra of mixed tensors
of the natural module and the dual module for the general linear superalgebra.
This theme is also pursued in \cite{RS}.
In the seminal paper \cite{H} on invariant theory of classical groups
which is ostensibly unrelated to supersymmetry,
Howe proved various `Howe dualities' among Lie superalgebras and ordinary Lie groups.
Similar results were obtained independently by Sergeev in \cite{S0}.
These Howe dualities have been fruitfully explored  in recent years to develop the
Segal-Shale-Weil representations of classical Lie superalgebras
(see \cite{CW, CZ, CLZ} and references therein). In particular, the $(\gl(m|n), \gl(k|l))$ Howe duality
survives quantisation \cite{WZ} and played a crucial role in understanding the q-deformed
Segal-Shale-Weil representations of the quantum general linear supergroup \cite{WZ, Z1, Z2}.
Further instances of surjective maps from the group ring of the braid group to endomorphism
algebras of tensor space may be found in \cite{LZ1, LZ2,LZ3,LZ4,LZZ}.

In this paper, we shall present a careful and elementary proof of the first fundamental theorem (FFT)
of invariant theory for the orthosymplectic supergroup, and derive from it a
super analogue of Schur-Weyl-Brauer duality between the supergroup and the Brauer algebra.
Our method is based on super-algebraic geometry, which permits one to convert invariants
of the orthosymplectic super group into those of the general linear supergroup.
In a future work we shall use our formulation to prove a second fundamental theorem
in this context.

Let $V_\C$ be a $\Z_2$-graded complex vector space of superdimension $(m|2n)$
endowed with a nondegenerate supersymmetric even bilinear form,
and let $\osp(V_\C)$ be the orthosymplectic Lie superalgebra of $V_\C$.
It has long been known that for any $r$, the Brauer algebra $B_r(d)$
with parameter $d=m-2n$ acts on $V_\C^{\otimes r}$,
respecting the $\osp(V_\C)$-action.
This was used in
decomposing low rank tensor powers of
$V_\C$ into simple $\osp(V_\C)$-modules in \cite{BLR}.
The authors of \cite{BLR} raised the question
whether the homomorphism from $B_r(d)$ to the
endomorphism algebra $\End_{\osp(V_\C)}(V_\C^{\otimes r})$ is a surjection.
It is not very difficult to see that the homomorphism cannot always be surjective.
We therefore consider the question of a
super Schur-Weyl-Brauer duality between the Brauer algebra
and the orthosymplectic supergroup, and that is
what we establishe here. In a future work, we shall describe
the extra invariants arising when considering the Lie superalgebra
$\osp(m|2n)$.

The orthosymplectic supergroup may be defined to be
the Harish-Chandra pair $({\rm O}((V_\C)_{\bar0})\times {\rm Sp}((V_\C)_{\bar1}), \osp(V_\C))$
following \cite{DM}.  An equivalent definition is as an algebraic group
over the Grassmann algebra $\Lambda$ of infinite degree
(defined as a direct limit). P. Deligne has pointed out that these two objects may be thought
of as the points of the superscheme $\OSp(V)$ over $\C$ and $\La$ respectively.
Let $V=V_\C\otimes_\C \Lambda$, equipped with the
supersymmetric form obtained by extending the form on
$V_\C$ bilinearly over $\Lambda$.
Then the orthosymplectic supergroup is defined to be the group $\OSp(V)$ of
invertible even endomorphisms of $V$ (as a $\Lambda$-module)
that preserve the supersymmetric form.

A large part of the paper is devoted to the proof of a linear version of the first fundamental theorem
of invariant theory for the orthosymplectic supergroup, that is, Theorem \ref{thm:fft-osp}.
A commutative superalgebra version of the FFT is deduced from it in Corollary \ref{cor:fft-osp-poly}.
The super Schur-Weyl-Brauer duality is the FFT in the setting
of endomorphism algebras. We present two versions of the duality, namely,
Corollaries \ref{cor:Brauer-super} and \ref{cor:Brauer},
applying respectively to the two incarnations above of the orthosymplectic supergroup.
A generalisation of the FFT for
the orthosymplectic supergroup to the setting of the Brauer
category of \cite{LZ5} is given in Theorem \ref{thm:fft-cat},
which contains both the linear and the endomorphism algebra versions
of the FFT as special cases.

Our proof of Theorem \ref{thm:fft-osp} follows in spirit the method
of Atiyah, Bott and Patodi in \cite[Appendix 1]{ABP} for proving the FFT
for the orthogonal group.
A detailed exposition of the method and its generalisation
to the symplectic group is given by Goodman and Wallach in \cite{GW1}.
The essential ingredient in our proof is the Key Lemma (i.e. Lemma \ref{lem:key},
or equivalently Lemma \ref{lem:key3}) established in Section \ref{sect:pf-new}.
It describes the $\OSp(V)$-invariant polynomial functions
on the even subspace of the endomorphism algebra of $V$. The proof of the
Key Lemma uses some geometric arguments in the graded-commutative context, as anticipated by Varadarajan
\cite[\S 4.5]{V} (see also \cite{DM}).

Given the Key Lemma, one reduces Theorem \ref{thm:fft-osp} to
the FFT for the general linear supergroup $\GL(V)$, i.e., Theorem \ref{thm:fft-gl}
(see also \cite[Theorem 1.1]{S1}), using a series of canonical, $\GL(V)$-module isomorphisms.
We prove Theorem \ref{thm:fft-gl} starting from the super Schur-Weyl duality
between $\gl(V_\C)$ and the symmetric group  \cite{BR, S0},  which is well established.
The $\GL(V)$-module isomorphisms needed are constructed
in Proposition \ref{prop:BW-OSp} in a way reminiscent of the Borel-Weil theorem.
Our proof of the FFT for the orthosymplectic
supergroup may be of interest in its own right, in that it provides an injective
map from invariants of $\OSp$ to invariants of $\GL$, whose invariant theory is
better understood and more straightforward.

We point out that Corollary \ref{cor:fft-osp-poly} is Sergeev's  \cite[Theorem 5.3]{S1}.
Our proof of the FFT is different from that given in \cite{S1}. In fact this work stems from
our attempt to understand {\it op.~cit.}

When $m$ is even and $r\ge m(2n+1)/2$, the endomorphism algebra
$\End_{\osp(V_\C)}(V_\C^{\otimes r})$ of $\osp(V_\C)$ contains elements which
are not invariants of the orthosymplectic supergroup.  Such elements arise from
the super analogue of the Pfaffian discovered by Sergeev \cite{S0} in the early 90s.
We explain a construction similar to Sergeev's in Section \ref{sect:pfaffian},
and briefly discuss how the super Pfaffian manifests itself in endomorphism algebras.
In a future work, we shall give a complete treatment of the invariants of $\osp(m|2n,\C)$
in this context.

Our formulation of classical supergroups is
in the spirit of the original physics literature \cite{SS}.
It is suitable for establishing algebraic results such as the
FFT and SFT of invariant theory for these supergroups.
One can define Lie supergroups \cite{DeW, Ko, M, V}
or algebraic supergroups \cite{DM, CCF} in more general ways,
but the simpler and more explicit formulation here permits a more concise treatment of the
fundamental theorems of invariant theory.
Finally we point out that our proof depends crucially on the infinite dimensional
nature of the direct limit Grassmann algebra $\Lambda$.

\medskip

\noindent{\bf Acknowledgement}.
The authors thank Pierre Deligne for discussions and correspondence about this work,
in particular for pointing out a gap in the proof of Lemma \ref{lem:key} in an earlier version.
We also thank Sasha Sergeev for discussions on super Schur-Weyl-Brauer duality.

\section{Linear superalgebra}
\subsection{$\Z_2$-graded vector spaces}\label{sect:basics}
We work over the field $\C$ of complex numbers. A $\Z_2$-graded vector space $V$ is
a direct sum $V_{\bar0}\oplus V_{\bar1}$ of two subspaces,
the even subspace $V_{\bar0}$ and odd subspace $V_{\bar1}$.
If $\dim V_{\bar0}=k$ and $\dim V_{\bar1}=l$, we write
$\sdim V=(k|l)$, and call it the superdimension of $V$.
Define the {\em parity} $[v]$ of any homogeneous element
$v\in  V_{\bar\alpha}$ by $[v]=\bar\alpha$. Parities are thought of
as elements of the ring $\Z_2=\Z/2\Z$, so may be added and multiplied.

If $V$ and $W$ are $\Z_2$-graded vector spaces,
their tensor product $V\otimes_\C W$
inherits a natural $\Z_2$-grading from the factors, with
\[
\begin{aligned}
(V\otimes_\C W)_{\bar 0}=V_{\bar 0}\otimes_\C W_{\bar 0}\oplus V_{\bar 1}\otimes_\C W_{\bar 1}, \\
(V\otimes_\C W)_{\bar 1}=V_{\bar 0}\otimes_\C W_{\bar 1}\oplus V_{\bar 1}\otimes_\C W_{\bar 0}.
\end{aligned}
\]
We define the permutation
\begin{equation}\label{eq:tau}
\tau_\C: V\otimes W \longrightarrow W\otimes V,
\end{equation}
which is a bilinear map sending $v\otimes w$ to $(-1)^{[v][w]}w\otimes v$
for any homogeneous elements $v\in V$ and $w\in W$.
\begin{remark}
We will extend expressions such as $(-1)^{[v][w]}w\otimes v$ to inhomogeneous
elements through linearity.
If $v=v_{\bar0}+v_{\bar1}$ and $w=w_{\bar0}+w_{\bar1}$ with
$v_{\bar i}\in V_{\bar i}$ and $w_{\bar i}\in W_{\bar i}$, then
$(-1)^{[v][w]}w\otimes v=w_{\bar0} \otimes v + w_{\bar1}\otimes(v_{\bar0}-v_{\bar1})=
w \otimes v_{\bar0} + (w_{\bar0}-w_{\bar1})\otimes v_{\bar1}$.
\end{remark}

The space $\Hom_\C(V, W)$ of homomorphisms is also $\Z_2$-graded with
\[
\begin{aligned}
\Hom_\C(V, W)_{\bar0}=\Hom_\C(V_{\bar0}, W_{\bar0})\oplus \Hom_\C(V_{\bar1}, W_{\bar1}),\\
\Hom_\C(V, W)_{\bar1}=\Hom_\C(V_{\bar0}, W_{\bar1})\oplus \Hom_\C(V_{\bar1}, W_{\bar0}).
\end{aligned}
\]
The $\Z_2$-graded dual vector space of $V$ is
$V^*:=\Hom_\C(V, \C)$. We denote $\Hom_\C(V, V)$ by $\End_\C(V)$.
Note that $\Hom_\C(V, W)\cong W\otimes_\C V^*$.

Whenever a basis is needed for a $\Z_2$-graded vector space,
we will always choose one which is homogeneous and ordered so that
the even basis elements precede the odd ones. We call such a basis an {\em
ordered homogeneous basis}. Assume that
$(v_1, \dots, v_m, v_{m+1}, \dots, v_{m+n})$ and $(w_1, \dots, w_k, w_{k+1}, \dots, w_{k+l})$
are such bases for
$V$ and for $W$ respectively, where $\sdim V=(m|n)$
and $\sdim W=(k|l)$.
Then $T\in \Hom_\C(V, W)$ acts on $V$ by
$T(v_b) = \sum_{a=1}^{k+l} w_a t_{a b}$.
Thus the map $T\mapsto(t_{a b})$ defines an
isomorphism between $\Hom_\C(V, W)$ and the $\Z_2$-graded vector space
$\cM(k|l\times m|n; \C)$ of $(k+l)\times(m+n)$ matrices, where the grading of $\cM(k|l\times m|n; \C)$
is defined as follows.
Write each matrix in the block form
$\begin{pmatrix} X & \Phi \\ \Psi & Y\end{pmatrix}$, where the sizes of
$X$, $\Phi$, $\Psi$ and $Y$ are respectively $k\times m$, $k\times n$, $l\times m$
and  $l\times n$. Then
\[
\cM(k|l\times m|n;\C)_{\bar0} = \left\{\begin{pmatrix} X & 0 \\ 0 & Y\end{pmatrix}\right\}, \qquad
\cM(k|l\times m|n;\C)_{\bar1} = \left\{\begin{pmatrix}0 & \Phi \\ \Psi & 0 \end{pmatrix}\right\}.
\]
We write $\cM(k|l;\C)$ for $\cM(k|l\times k|l;\C)$.

\subsection{Linear superalgebra}\label{sect:lin-alg}

Given a vector space $\Theta$, we denote by $T(\Theta)$ its tensor algebra, which is
$\Z_+$-graded with degree $1$ subspace $T^1(\Theta)=\Theta$.
Let $I$ be the 2-sided ideal of $T(\Theta)$ generated by the elements
$x y + y x$ for all $x, y\in \Theta$.  Since these elements are homogeneous, the ideal $I$ is graded.
The exterior algebra over $\Theta$ is the $\Z_+$-graded algebra defined by
$
\Lambda(\Theta) = T(\Theta)/I.
$
As a complex vector space, $\Lambda(\Theta)$ has dimension $2^N$, where $N=\dim\Theta$.
We endow $\Lambda(\Theta)$ with the $\Z_2$-grading in which
$\Lambda_{\bar0}$ (resp. $\Lambda_{\bar1}$)
is the direct sum of the homogeneous subspaces of even (resp. odd) degrees.
This superalgebra is called the {\em Grassmann algebra} of degree $N$
in the physics literature.
Choose any basis $\theta_i$ $(i=1, 2, \dots, N)$ for $\Theta$. Then we have
the following basis for $\Lambda(\Theta)$: for any sequence $A=(\alpha_1,\dots,\alpha_N)$ with
$\alpha_i=0, 1$, define
\begin{eqnarray}\label{eq:grass}
\theta^{A}=\theta_1^{\alpha_1} \theta_2^{\alpha_2}\dots\theta_N^{\alpha_N} \quad.
\end{eqnarray}
The $\theta^A$ form a basis of $\Lambda(\Theta)$.

Clearly, up to isomorphism, $\Lambda(\Theta)$ depends only on the degree $N=\dim\Theta$, and we therefore write
$\Lambda(N)$ for $\Lambda(\Theta)$ when $\dim \Theta= N$. For any flag
$0\subset\Theta^{(1)}\subset\Theta^{(2)}\subset\dots\subset\Theta^{(N-1)}\subset\Theta^{(N)}=\Theta$,
where $\dim\Theta^{(i)}=i$, we have the corresponding inclusions
$\C\subset\Lambda(1)\subset\Lambda(2)\subset\dots\subset\Lambda(N-1)\subset\Lambda(N)$
of Grassmann algebras. The Grassmann algebra
$\Lambda$ of infinite degree is defined by the direct limit of this direct system,
$\Lambda:=\lim\limits_{\longrightarrow}\Lambda(N).$
Besides the $\Z_2$-grading, $\Lambda$ has a $\Z_+$-grading $\Lambda=\oplus_{i=0}^\infty \Lambda_i$,
where $\Lambda_0=\C$. Write
$\Lambda_+=\oplus_{i\ge 1}\Lambda_i$.
For any $\lambda\in\Lambda$, denote by $\lambda_0$ its degree $0$ component.
Then we have the specialisation, or augmentation map,
\begin{eqnarray}\label{eq:specialisation}
\cR: \Lambda\longrightarrow\C, \quad \lambda\mapsto \lambda_0,
\end{eqnarray}
which is simply the projection of $\Lambda$ onto $\Lambda_0=\C$.

Throughout this paper,  all the (left, right and bi-) modules over Grassmann algebras
are assumed to be $\Z_2$-graded.
Since Grassmann algebras are super commutative  (i.e., $Z_2$-graded commutative),
any right (or left) module is naturally a bimodule.
For example, any right $\Lambda$-module $V_\Lambda$ can be turned
into a $\Lambda$-bimodule with a compatible left action defined by the composition
$
\Lambda\otimes_\C V_\Lambda\stackrel{\tau_\C}{\longrightarrow}
V_\Lambda\otimes_\C\Lambda\longrightarrow V_\Lambda,
$
where $\tau_\C$ is the permutation \eqref{eq:tau} in the present context,
and the second map is the right action.  In particular, we have
$\lambda v = (-1)^{[\lambda][v]} v\lambda$ for $\lambda\in\Lambda$ and $v\in V_\Lambda$.

\begin{remark}
Henceforth the term $\Lambda$-module will mean a $\Z_2$-graded $\Lambda$-bimodule.
\end{remark}

A map $T: V\longrightarrow W$ between $\Lambda$-modules $V$ and $W$
is homogeneous of degree $\bar\alpha$ if $T(V_{\bar\beta})\subset
W_{\overline{\alpha+\beta}}$ for all $\bar\beta$. We say that $T$
is $\Lambda$-linear if
$T(v \lambda)=T(v)\lambda$ for all $v\in V_\Lambda$ and $\lambda\in\Lambda$.
In particular this implies $T(\lambda v)=(-1)^{[\lambda][T]}\lambda T(v)$.
Denote by $\Hom_\Lambda(V, W)$ the $\Z_2$-graded vector space of
$\Lambda$-linear maps, and set $\End_\Lambda(V):=\Hom_\Lambda(V, V)$
and $V^*_\Lambda:=\Hom_\Lambda(V, \Lambda)$.
Note also that $\Hom_\Lambda(V, W)$ has a natural $\Lambda$-module structure, defined
for all $T\in\Hom_\Lambda(V, W)$ and $\lambda\in\Lambda$ by
\begin{equation}\label{eq:lambda-module}
\begin{aligned}
T\lambda=(-1)^{[\lambda][T]} \lambda T, &\quad & (\lambda T)(v)= \lambda T(v),&\quad
T\la(v)=(-1)^{[\lambda][v]}T(v)\lambda \quad
\end{aligned}
\end{equation}
for all  $T\in\Hom_\La(V,W)$ and $v\in V$.

For any $\Z_2$-graded vector space $V_\C$, we let $V(\Lambda)=V_\C\otimes_\C\Lambda$
be the corresponding free $\Lambda$-module.  We have the natural embedding
$V_\C\hookrightarrow V(\Lambda)$, $v\mapsto v\otimes 1$.
If $W_\C$ is another $\Z_2$-graded vector space, then
$
\Hom_\Lambda(V(\Lambda), W(\Lambda)) \cong \Hom_\C(V_\C, W_\C)\otimes_\C\Lambda.
$

Let $\cM(k|l\times m|n; \Lambda)$ denote the set of
$(k+l)\times (m+n)$ matrices over $\Lambda$, and
write $\cM(k|l; \Lambda)$ for $\cM(k|l\times k|l; \Lambda)$.
Then $\cM(k|l\times m|n; \Lambda)\cong\cM(k|l\times m|n;\C)\otimes_\C \Lambda$
as $\Z_2$-graded vector space.
We write a matrix $\tilde{T}$ in block form as
$\tilde{T}=\begin{pmatrix} A & B \\ C & D\end{pmatrix}$.
If $\tilde{T}$ is even (resp. odd), then the entries of $A$ and $D$ belong to $\Lambda_{\bar0}$
(resp. $\Lambda_{\bar1}$), and
the entries of $B$ and $C$ belong to $\Lambda_{\bar1}$ (resp. $\Lambda_{\bar0}$).

For any  $T \in \Hom_\Lambda(V(\Lambda), W(\Lambda))$, we write
$
T(v_b) = \sum_{a=1}^{k+l} w_a t_{a b}, \quad \text{$t_{a b}\in\Lambda$},
$
where $(v_a)$ and $(w_b)$ are homogeneous bases of $V_\C$ and $W_\C$ respectively.
This leads to the following isomorphism of $\Z_2$-graded vector spaces.
\begin{equation}\label{eq:matrices}
\Upsilon: \Hom_\Lambda(V(\Lambda), W(\Lambda))
\stackrel{\simeq}\longrightarrow \cM(k|l\times m|n; \Lambda),  \quad T\mapsto(t_{a b}).
\end{equation}
However, the $\Lambda$-module structure of
$\cM(k|l\times m|n; \Lambda)$ needs to be treated with care.
Define $\C$-linear maps
\begin{equation}\label{eq:act-mat}
\begin{aligned}
\triangleright: \Lambda \otimes \cM(k|l\times m|n; \Lambda) \longrightarrow \cM(k|l\times m|n; \Lambda), \\
\triangleleft: \cM(k|l\times m|n; \Lambda) \otimes\Lambda  \longrightarrow \cM(k|l\times m|n; \Lambda)
\end{aligned}
\end{equation}
respectively by
\begin{equation}\label{eq:act-mat1}
\lambda\triangleright M =  \left((-1)^{[w_a][\lambda]}\lambda M_{a b}\right),
\quad M \triangleleft\lambda =   \left((-1)^{[\lambda][v_b]}  M_{a b}\lambda\right),
\end{equation}
for all $M=(M_{a b})\in \cM(k|l\times m|n; \Lambda)$ and $\lambda\in\Lambda$.
The following easy lemma may be found in \cite[\S A.4]{SZ}.
\begin{lemma}
\begin{enumerate}\label{lem:superlin}
\item Given $T\in\Hom_\Lambda(V(\Lambda),W(\Lambda)$ define $\cM(T)\in  \cM(k|l\times m|n; \Lambda)$
by $\cM(T)=(t_{ab})$, where $Tv_b=\sum_a w_a t_{ab}$. Then $\cM(\la T) =\la\triangleright\cM(T)$ and
$\cM(T\la) =\cM(T)\triangleleft\la$.
\item \label{lem:module-mat}
Equation \eqref{eq:act-mat} defines a compatible $\Lambda$-bimodule structure
on $\cM(k|l\times m|n; \Lambda)$.
\item \label{lem:module-isom}
The map $\Upsilon$ defined by \eqref{eq:matrices}
is an isomorphism of $\Lambda$-bimodules, where the left and right $\Lambda$-actions on
$\Hom_\Lambda(V(\Lambda), W(\Lambda))$ are defined by \eqref{eq:lambda-module},
and those on $\cM(k|l\times m|n; \Lambda)$ by \eqref{eq:act-mat}.
\item When $W=V$, the map $\Upsilon$ is an isomorphism
of associative $\Lambda$-superalgebras, where multiplication in $\cM(k|l\times k|l;\Lambda)$ is
just multiplication of matrices.
\end{enumerate}
\end{lemma}
All statements follow from (1), which is an easy calculation.
\subsection{Orthosymplectic superspace}\label{sect:orthosym}

Let $V_\C$ be a complex $\Z_2$-graded vector space.
We assume that $V_\C$ admits a non-degenerate even bilinear form
\[
(-, - )_\C: V_\C\times V_\C \longrightarrow \C,
\]
which is supersymmetric, that is,  $( u  , v )_\C =(-1)^{[u][v]} ( v,  u )_\C$
for all $u, v\in V_\C$.   Then the form is
symmetric on $(V_\C)_{\bar0}\times (V_\C)_{\bar0}$ and
skew symmetric on $(V_\C)_{\bar1}\times (V_\C)_{\bar1}$, and
satisfies $(V_{\bar0},  V_{\bar1})_\C=0=(V_{\bar1}, V_{\bar0})_\C$.
Also, by non-degeneracy, $\dim (V_\C)_{\bar1}=2n$ must be even.
We call this a nondegenerate {\em supersymmetric  form}.   Let
$
\eta = \begin{pmatrix} I_m & 0 \\ 0 & J\end{pmatrix},
$
where $I_m$ is the identity matrix of size $m\times m$ and $J$ is a skew symmetric matrix of size
$2n\times 2n$ given by
$J=diag(\sigma, \dots, \sigma)$ with $\sigma=\begin{pmatrix} 0 & -1 \\ 1 & 0\end{pmatrix}$.
Then there exists an ordered homogeneous basis $\cE=(e_1,  e_2, \dots, e_{m+2n}) $ for $V_\C$
such that
\begin{eqnarray}\label{eq:standard}
(e_a, e_b)=\eta_{a b}, \quad \text{for all $a, b$}.
\end{eqnarray}
We shall write such relations as $(\cE,\cE)=\eta$.

Let $V:=V_\C\otimes\Lambda$. We extend $(\ . \ , \ . )_\C$ to a $\Lambda$-bilinear form
$
( - , - ): V\times V\longrightarrow \Lambda,
$
which is even and nongenerate,  and is supersymmetric in the sense that
$(v, w)=(-1)^{[v][w]}(w, v)$ for all $v, w\in V$. We call $V$ an
orthosymplectic superspace, and call $\cE$
an {\em orthosymplectic basis} of $V_\C$ and of $V$.
Note that $(.,.)$ is $\Lambda$-bilinear in the sense that
for $\la,\la'\in\La$ and $v,v'\in V$, we have
\be\label{eq:bil}
(\la v, v'\la')=\la(v,v')\la'.
\ee

The following result generalises the Gram-Schmidt process to superspaces.
\begin{lemma}\label{lem:bs-constr}
 Let $V:=V_\C\otimes\Lambda$ be an orthosymplectic superspace
 of superdimension $\sdim V_\C=(m|2n)$.
\begin{enumerate}
\item Assume that $u$ belongs to the even subspace $V_{\bar0}$ of $V$ and $(u, u)=1$.
Then there exists an ordered homogeneous basis $B=(b_1, b_2, \dots, b_{m+2n})$
of the $\Lambda$-module $V$ such that $b_1=u$ and $(B, B)=\eta$.
\item Assume that $v, v'\in V_{\bar1}$ and $(v, v')=1$. Then there exists an ordered homogeneous basis
$C=(c_1, c_2, \dots, c_{m+2n})$ of the $\Lambda$-module $V$ such that $c_{m+2n-1}=v$
and $c_{m+2n}=v'$ and $(C, C)=\eta$.
\end{enumerate}
\end{lemma}
The bases above are of course not unique.
\begin{proof} We prove part (1).  Any element $w \in V$ can be written
uniquely as a finite sum $w=w_0 + w_1 + w_2 + \dots$,
where $w_j\in V_\C\otimes_\C \Lambda_j$ with $ \Lambda_j$ being
the degree $j$ homogeneous subspace of $\Lambda$.
Expressing $u$ this way, we have $u_0\in (V_\C)_{\bar 0}$ and $(u_0,u_0)=1$. By a standard result in linear algebra,
there is a sequence $B'_0=(v^{(1)}_0, v^{(2)}_0, \dots, v^{(m+2n-1)}_0)$ of homogeneous elements of $V_\C$
such that $B_0:=(u_0, B'_0)$ forms an orthosymplectic basis of $V_\C$, that is $(B_0, B_0)=\eta$.
We shall show that there exists a sequence $B'=(v^{(1)},v^{(2)},\dots)$ of vectors in $V$ with degree zero specialisation
$B'_0$ such that the sequence $B:=(u,B')$ satisfies $(B, B)=\eta$.
Given any such sequence $B$,
let $g$ be the corresponding transition matrix from the basis $B_0$, so that $B=B_0 g$.
Since $B$ is finite, there is an integer $N$ such that all elements of $B$ belong to $V_\C\otimes\Lambda(N)$,
whence $g=1+\Theta$ for some even matrix $\Theta$ with entries belonging to
$\Lambda(N)_+=\sum_{i\ge 1} \Lambda(N)_i$.
Since $\Theta$ is nilpotent, $g$ is invertible, and so $B$ is a basis of $V$.

We now turn to the problem of finding such a sequence $B'=(v^{(1)},\dots,v^{(m+2n-1)})$.
Regard the higher degree terms $v^{(k)}_i$ ($i>1$) as unknowns.

Then $B=(u, B')$ satisfies $(B, B)=\eta$ if and only if
\begin{eqnarray*}
\begin{aligned} \sum_{i+j=p}( u_j, v^{(k)}_i)=0, \quad
\sum_{i+j=p}(v^{(k)}_i, v^{(\ell)}_j)=0, \quad \forall p\ge 1,\forall k,\ell\ge 1
\end{aligned}
\end{eqnarray*}
using the supersymmetry of the form $(\ , \ )$. Now re-arrange the terms of the equations to obtain,
for all $p\ge 1$ and for all $k,\ell$,
\begin{eqnarray}\label{eq:orthosym}
(u_0, v^{(k)}_p)&=- \sum_{i=1}^p(u_i, v^{(k)}_{p-i})\\ \label{eq:orthosym2}
 (v^{(k)}_0, v^{(\ell)}_p) +(v^{(k)}_p, v^{(\ell)}_0)
&= -\sum_{i=1}^{p-1}(v^{(k)}_i, v^{(\ell)}_{p-i}),
\end{eqnarray}
where we note that only $v^{(s)}_j$ with $j<p$ are present on the right sides of these equations.
Write $\bv_p=(v^{(1)}_p,v^{(2)}_p,\dots,v^{(m+2n-1)}_p)$, so that $\bv_0=B'_0$, and our task is to
solve the above equations for $\bv_p$, $p\geq 1$.
Since $B_0$ is a basis of $V_\C$, each ${\bf v}_p$  ($p\ge 1$) can be expressed as
$
{\bf v}_p = u_0 r_p + {\bf v}_0 T_p
$
for some $r_p\in \Lambda_p$ and $(m+2n-1)\times(m+2n-1)$ matrix
$T_p$ with entries in $\Lambda_p$.
Then equation \eqref{eq:orthosym} is equivalent to
$
r_p = - \sum_{i=1}^p(u_i, {\mathbf v}_{p-i}),
$
and hence
\begin{eqnarray}\label{eq:iterate}
{\bf v}_p =- u_0\sum_{i=1}^p(u_i, {\mathbf v}_{p-i}) + {\bf v}_0 T_p,  \quad p\ge 1.
\end{eqnarray}
Further, equation  \eqref{eq:orthosym2} is equivalent to
\[
(\eta' T_p)_{a b} +(-1)^{[a][b]}(\eta' T_p)_{b a}= - \sum_{i=1}^{p-1}(v^{(a)}_i, v^{(b)}_{p-i}),  \quad p\ge 1,
\]
where $\eta'=\begin{pmatrix}I_{m-1}&0\\ 0&J\end{pmatrix}$.
Using the supersymmetry of the form, we can re-write the right hand side as
$
  -\frac{1}{2} \left( \sum_{i=1}^{p-1}(v^{(a)}_i, v^{(b)}_{p-i})
+ (-1)^{[a][b]}\sum_{i=1}^{p-1}(v^{(b)}_i, v^{(a)}_{p-i})\right).
$

This shows that the equation has an evident (though not unique) solution
\begin{eqnarray}\label{eq:T}
(\eta' T_p)_{a b} =  -\frac{1}{2}\sum_{i=1}^{p-1}(v^{(a)}_i, v^{(b)}_{p-i}),     \quad p\ge 1,
\end{eqnarray}
which expresses $T_p$ in terms of ${\mathbf v}_j$ with $j<p$.
Substituting \eqref{eq:T} into \eqref{eq:iterate} we obtain a recursion relation which enables us to determine
the coefficients of each ${\mathbf v}_p$ with $p\ge 1$ with respect to the basis $u_0,{\mathbf v}_0$
in terms of $\bv_{p'}$ for $p'<p$.

It remains only to show that this process terminates, i.e. that ${\mathbf v_p=0}$ for $p>>0$.
For this, note that there exists some $N$ so large that  $u\in V_\C\otimes \Lambda(N)$,
and hence ${\mathbf v}_p\in V_\C\otimes \Lambda(N)$ for all $p$. Thus ${\mathbf v}_p=0$ for $p>>0$.
This proves part (1) of the lemma.

The proof of part (2) is similar and we omit the details.
\end{proof}

It is easily seen that the proof of the above lemma may be applied to prove the following more general result.

\begin{scholium}\label{sch:gs}
Let $K\supseteq \C$ be any field
such that the statement of Lemma \ref{lem:bs-constr} holds for $V(K):=V_\C\ot_\C K$.
Then Lemma \ref{lem:bs-constr} holds for $V_K:=V(K)\ot_\C\La\cong V_\C\ot_\C\La_K$, where
$\Lambda_K:=\La\ot_\C K\cong\lim\limits_{\longrightarrow}\Lambda_K(N)$ with $\La_K(N)=\wedge(K^N)$.
\end{scholium}
\begin{proof}
The proof of (1) proceeds exactly as above, with the first step following from the given hypothesis on $K$.
\end{proof}
\subsection{Classical supergroups}\label{sect:groups}

Let $V_\C$ be a $\Z_2$-graded vector space of superdimension $\sdim V_\C=(m|l)$.
As usual, we write
\[
V=V_\C\otimes_\C\Lambda, \quad V^*=\Hom_\Lambda(V, \Lambda)\cong V^*_\C\otimes_\C\Lambda.
\]

\subsubsection{The general linear supergroup}

The general linear Lie superalgebra $\gl(V_\C)$ over $\C$ is
$\End_\C(V_\C)$ endowed with a bilinear
Lie bracket defined for $X, Y\in \gl(V_\C)$ by
\[
[X, Y]=X Y - (-1)^{[X][Y]}Y X,
\]
where the right hand side is defined by composition of endomorphisms.
The general linear supergroup
$\GL(V)$ is the set
\[
\GL(V)=\{g\in \End_{\Lambda}(V)_{\bar0}\mid g \ \text{ is invertible}\}
\]
with multiplication defined by the composition of endomorphisms on $V$.

Let $\widetilde\gl(V)=\gl(V_\C)\otimes_\C\Lambda$ and regard it as a Lie superalgebra
over $\Lambda$ with a $\Lambda$-bilinear Lie bracket defined by
\[
[X\otimes\lambda, Y\otimes\mu]=[X, Y]\otimes (-1)^{[\lambda]([Y\otimes\mu])}\mu\lambda,
\]
for all $X, Y\in \gl(V_\C)$ and $\mu, \nu\in\Lambda$. Then $\gl(V)=(\gl(V_\C)\otimes_\C\Lambda)_{\bar0}$
forms a Lie subalgebra of $\widetilde\gl(V)$ over $\Lambda_{\bar 0}$, which will be referred to as the
Lie algebra of $\GL(V)$.
There is a natural $\widetilde\gl(V)$ action on $V$ given by
$(X\otimes\lambda).(v\otimes\mu)=
X.v\otimes (-1)^{[\lambda]([v\otimes\mu])}\mu\lambda$.
It restricts to an action of $\gl(V)$.

The following lemma is a generalisation of \cite[Proposition 5.2 (a), (c)]{SZ} to the Grassmann algebra
of infinite degree.
\begin{lemma}\label{lem:exp}
\begin{enumerate}
\item Given any $X\in \gl(V)$, let
$\exp(X):=\sum_{i=0}^\infty \frac{X^i}{i!}$.
Then $\exp(X)$  is a well defined automorphism of $V$ which lies in $\GL(V)$.
Hence there exists a map
\[
{\rm Exp}: \gl(V) \longrightarrow \GL(V), \quad X\mapsto \exp(X).
\]
\item The image of $\Exp$ generates $\GL(V)$.
\end{enumerate}
\end{lemma}
\begin{proof}
(1) For any fixed $X\in \gl(V)$, there exists some sufficiently large but finite $N$
such that $X\in \End_\C(V_\C)\otimes \Lambda_N$.  Then the relevant part of the proof of
\cite[Proposition 5.2]{SZ} goes through, and $\exp(X)$ is well defined in
$\End_\C(V_\C)\otimes \Lambda_N\subset \End_{\Lambda}(V)$.
Clearly $\exp(X)$ is even and has inverse $\exp(-X)$.

(2) Any element $g\in\GL(V)$ may be written as $g=g_0(1+x)$, where $g_0\in\GL(V_\C)$
and $x\in \End(V_\C)\ot\La_{\geq 1}$. It therefore suffices to show that both $g_0$ and
$1+x$ are in the group generated by the image of $\Exp$.
But for $g_0$ this follows from a standard result in
complex Lie groups. Moreover since $x$ is nilpotent, the series
$\log(1+x)=\sum_{i=1}^\infty(-1)^{i-1}\frac{x^i}{i}$
 terminates, and is equal to (say) $y\in\gl(V)$. Then $1+x=\Exp(y)$,
and we are done.
\end{proof}

\subsubsection{The orthosymplectic supergroup}\label{sect:orthosym-gp}

Suppose that $V:=V_\C\otimes\Lambda$ is an orthosymplectic superspace.
Then the orthosymplectic supergroup of $V$ is the subgroup
\[
\OSp(V):=\{g\in\GL(V)\mid (g v, g w)=(v, w), \  \forall v, w\in V\}
\]
of $\GL(V)$.
The Lie algebra of $\OSp(V)$ is the Lie algebra over $\Lambda_{\bar0}$ defined by
\[
\osp(V)=\{A\in\End_\Lambda(V)_{\bar0}\mid (A v, w)+(v, Aw)=0,\  \forall v, w\in V\}.
\]
It is easily seen that for $A\in \End_\Lambda(V)_{\bar0}$, we have $A\in\osp(V)$
if and only if $\exp(tA)\in\OSp(V)$ for all $t\in\C$.

Let $\osp(V_\C)$ be the orthosymplectic Lie superalgebra over $\C$ \cite{K},
which is the Lie sub-superalgebra of $\gl(V_\C)$  given by
\[
\osp(V_\C) =\{X\in \End_\C(V_\C) \mid (X v, w)_\C + (-1)^{[X][v]}(v, X w)_\C = 0,\  \forall v, w\in V_\C\}.
\]
Then $\osp(V)=(\osp(V_\C)\otimes\Lambda)_{\bar0}$.

Write $\OSp(V)_0={\rm O}((V_\C)_{\bar0})\times{\rm Sp}((V_\C)_{\bar1})$;
this group, which could be thought of as the degree zero part of $\OSp(V)$,
acts on $\osp(V_\C)$ by conjugation, and
both $\OSp(V)_0$ and $\osp(V_\C)$ act  naturally on $V_\C^{\otimes r}$ for all $r$.
Their actions are compatible in the following sense. For all $g\in\OSp(V)_0$,
$X\in \osp(V_\C)$ and $w\in V_\C^{\otimes r}$, we have
$$
g(X w) =Ad_g(X)(g w),
$$
where $Ad_g$ denotes the conjugation action of $g\in\OSp(V)_0$ on $\osp(V_C)$.

\begin{remark}\label{rem:HC-pair}
Note that $(\OSp(V)_0, \osp(V_\C))$ is a Harish-Chandra pair;
this could be used to give an alternative definition of  the orthosymplectic supergroup $\OSp(V)$
(see \cite{DM} for details).
\end{remark}

The following result is an immediate consequence of the definition of $\OSp(V)$.

\begin{lemma}\label{lem:bs-change} Let $\sdim V_\C=(m|2n)$, and
let $B=(b_1, b_2, \dots, b_{m+2n})$ be a homogeneous $\Lambda$-basis of $V$ such that
$b_i$ is even (resp. odd) if $i\le m$ (resp. $i>m$). If $(B, B)=\eta$,
there exists a unique element $g\in\OSp(V)$ such that
$$B=(g e_1, g e_2, \dots, g e_{m+2n}),$$ where
$(e_1, e_2, \dots, e_{m+2n})$ is the basis \eqref{eq:standard}.
\end{lemma}
\begin{proof} Given the existence of $g$, uniqueness is trivial, so we prove existence.
Write $B=B_0+B_1+B_2+\dots$ as in the proof of Lemma \ref{lem:bs-constr}, so that
$B_p$ is the degree $p$ component of $B$, whose entries belong to
$V_\C\otimes \Lambda_p$. Then $B_0=(b_{1, 0}, b_{2, 0}, \dots, b_{m+2n, 0})$ satisfies $(B_0, B_0)=\eta$.
Therefore, there exists an even matrix $M'$ such that $B=B_0 M'$,
i.e., $b_i=\sum_{j} b_{j, 0} M'_{j i}$ for all $i$.  Let $g'\in \End_\Lambda(V)_{\bar0}$  be such that
$(g' b_{1, 0}, g' b_{2, 0}, \dots, g' b_{m+2n, 0})=B_0 M'$. Then it is clear that $g'\in\OSp(V)$.
Further, $B_0=(g_0 e_1, g_0 e_2, \dots, g_0 e_{m+2n})$ for some
element $g_0\in\OSp(V)_0$ by standard linear algebra. Hence $g=g_0 g'$ satisfies the required conditions.
\end{proof}

\section{Invariant theory for the general linear supergroup}\label{sect:inv-gl}
\subsection{First fundamental theorem}

Let $V^{\otimes_\Lambda r}=V\otimes_\Lambda \dots \otimes_\Lambda V$ ($r$ factors)
and write $T^r(V)= V^{\otimes_\Lambda r}$.  The general linear supergroup $\GL(V)$ acts on $T^r(V)$ by
$
g.w=g w_1\otimes...\otimes g w_r
$
for any $w=w_1\otimes...\otimes w_r$ and $g\in \GL(V)$.
The corresponding $\gl(V)$-action on $V^{\otimes_\Lambda r}$ is defined for all $X\in\gl(V)$ by
\be\label{eq:glvaction}
\begin{aligned}
X.w& = Xw_1\otimes w_2\otimes  \dots \otimes  w_r+ w_1\otimes X w_2\otimes  \dots \otimes  w_r\\
&+\dots + w_1\otimes w_2\otimes  \dots \otimes X w_r.
\end{aligned}
\ee
We denote the associated representations of both $\GL(V)$ and $\fg(V)$ on  $V^{\otimes_\Lambda r}$
by $\rho_r$.

The permutation map $\tau$ of \eqref{eq:tau} extends uniquely to a $\Lambda$-module map
\[
\tau: V\otimes_\Lambda V \longrightarrow
V\otimes_\Lambda V, \quad v\otimes w\mapsto (-1)^{[v][w]}w\otimes v.
\]
This defines a $\Z_2$-action on $V\otimes_\Lambda V$.
More generally, we define a $\Lambda$-linear action $\varpi_r$
of the symmetric group $\Sym_r$ of degree $r$ on $T^r(V)$ as follows. If $s_i$, $1\le i\le r-1$,
are the simple reflections which generate  $\Sym_r$, then for all $i$, we define $\varpi_r(s_i)$ by
\begin{eqnarray}\label{eq:sym}
\varpi_r(s_i): w \mapsto  w_1\otimes\dots \otimes \tau( w_i\otimes w_{i+1})\otimes \dots \otimes  w_r.
\end{eqnarray}
The group ring $\Lambda\Sym_r=\C\Sym_r\otimes_{\C}\Lambda$ is an associative superalgebra,
with $\C\Sym_r$ regarded as purely even. Extend the representation $\varpi_r$ of $\Sym_r$
$\Lambda$-linearly  to obtain an action of the super algebra $\Lambda\Sym_r$.

It is easy but important to observe that the actions of $\Lambda\Sym_r$ and  $\GL(V)$ on $T^r(V)$
commute with each other.
\begin{remark}\label{rem:ospclass}
The classical orthosymplectic Lie algebra $\osp(V_\C)$ acts on $V_\C^{\ot r}$ via the rule
$$
X.(w_1\ot\dots\ot w_r)=\sum_{i=1}^r(-1)^{[X]([w_1]+[w_2]+\dots[w_{i-1}])}w_1\ot\dots\ot  Xw_i\ot\dots\ot w_r.
$$
This action is compatible with that of $\gl(V)$ \eqref{eq:glvaction}. Moreover it is easily verified
that $\varpi(\C\Sym_r)$ commutes with $\gl(V_\C)$ on $V_\C^{\ot r}$.
\end{remark}

The following well-known result \cite{S1, BR}, is the super analogue of Schur-Weyl duality.
\begin{theorem}\label{thm:BR} Let $\End_{\GL(V)}(V^{\otimes_\Lambda r})
= \{\phi\in\End_\Lambda(V^{\otimes_\Lambda r}) \mid g\phi = \phi g, \ \forall g\in \GL(V) \}$. Then
$\End_{\GL(V)}(V^{\otimes_\Lambda r}) = \varpi_r(\Lambda\Sym_r)$.
\end{theorem}

\begin{proof} We provide a proof for the convenience of the reader. By Lemma \ref{lem:exp}(2),
 it suffices to prove that
\begin{equation}\label{eq:lie-alg}
\End_{\gl(V)}(V^{\otimes_\Lambda r}) = \varpi_r(\Lambda\Sym_r).
\end{equation}
We therefore turn to the proof of \eqref{eq:lie-alg}.
It is well known  \cite{BR} that a generalised form of Schur-Weyl duality holds for the complex
general linear superalgebra $\gl(V_\C)$, that is,
\begin{eqnarray}\label{eq:SW-C}
\End_{\gl(V_\C)}(V_\C^{\otimes r}) = \varpi_{\C, r}(\C\Sym_r),
\end{eqnarray}
where $\varpi_{\C, r}$ is the representation
of $\C\Sym_r$ on $V_\C^{\otimes r}$ defined in the same way as \eqref{eq:sym}, and
$\End_{\gl(V_\C)}(V_\C^{\otimes r}) =\{\phi\in \End_{\C}(V_\C^{\otimes r}) \mid X \phi
= (-1)^{[\phi][X]} \phi X, \  \forall X\in \gl(V_\C)\}$.  Since $\varpi_r(\Lambda\Sym_r)
=\varpi_{\C, r}(\C\Sym_r)\otimes\Lambda$, it suffices to show that $\End_{\gl(V)}(V^{\otimes_\Lambda r})=
\End_{\gl(V_\C)}(V_\C^{\otimes r})\otimes\Lambda$
in order to prove \eqref{eq:lie-alg}.

Now $\End_{\Lambda}(V^{\otimes_\Lambda r})
=\End_{\C}(V_\C^{\otimes r})\otimes\Lambda$, thus every $\phi\in\End_{\Lambda}(V^{\otimes_\Lambda r})$
can be expressed in the form $\phi=\sum_i \phi_i\otimes\lambda_i$ with $\phi_i\in\End_{\C}(V_\C^{\otimes r})$
and $\lambda_i\in\Lambda$, where we may take the $\lambda_i$ to be $\C$-linearly independent and homogeneous.
If $\phi\circ (X\otimes\lambda)- (X\otimes\lambda)\circ\phi=0$ for all $X\in \gl(V_\C)$ and all $\lambda\in\Lambda$
such that $[X]=[\lambda]$, we have
\[
\sum_i \left(\phi_i\circ X -(-1)^{[X][\phi_i]}X\circ\phi_i \right)\otimes (-1)^{[X][\lambda_i]}\lambda_i\lambda =0.
\]
Since $\Lambda$ is of infinite degree, the above equation holds if and only if
\begin{eqnarray}
\sum_i \left(\phi_i\circ X -(-1)^{[X][\phi_i]}X\circ\phi_i \right)\otimes (-1)^{[X][\lambda_i]}\lambda_i=0.
\end{eqnarray}
This is equivalent to $\phi_i\circ X -(-1)^{[X][\phi_i]}X\circ\phi_i=0$ for all $i$
since the $\lambda_i$ are linearly independent over $\C$.
That is $\phi_i\in \End_{\gl(V_\C)}(V_\C^{\otimes r})$ for all $i$.
This completes the proof.
\end{proof}

\begin{remark}
It is crucial for the above statement that $\Lambda$ has infinite degree.
If $\Lambda$ were replaced by $\Lambda(N)$ for finite $N$, then for any $t\ne 0$
of top degree and any odd $\lambda$ in $\Lambda(N)$, we would have $t\lambda=0$.
It would then follow that $\phi\otimes t$ is a $\GL(V)$-invariant for any
$\phi\in\End_{\gl(V_\C)_{\bar0}}(V_\C^{\otimes r})$, where
$\gl(V_\C)_{\bar0}=\gl((V_\C)_{\bar0})\oplus\gl((V_\C)_{\bar1})$ is the even subalgebra
of $\gl(V_\C)$.
\end{remark}

\begin{remark}
Theorem \ref{thm:BR}  is one side
of a double commutant theorem for $\gl(V_\C)$; that is, we also have
$
\End_{\Sym_r}(V_\C^{\otimes r}) = \rho_r(U(\gl(V_\C)))
$
in addition to \eqref{eq:SW-C}. This relies on the semi-simplicity of
$V_\C^{\otimes r}$ as a $\gl(V_\C)$-module for all $r$. This semi-simplicity
may be proved  \cite{GZ} by showing that the natural module $V_\C$
is unitarisable, which implies that the tensor modules
$V_\C^{\otimes r}$ are unitarisable.
\end{remark}

\begin{remark}
One can deduce from \eqref{eq:SW-C} that
$\End_{\gl(V_\C)}(V_\C^{\otimes r}\otimes {V_\C^*}^{\otimes s})$ is a quotient of the walled Brauer algebra
with parameter $m-n$ (see \cite{BS} for details).
\end{remark}

The dual superspace $V^*$ has a natural $\GL(V)$-module structure. For any
$\bar{v}\in V^*$ and $g\in \GL(V)$, $g.\bar{v}$ is defined by
$g.\bar{v}(w)= \bar{v}(g^{-1}w)$ for all $w\in V.$
This extends to $T^r(V^*):={V^*}^{\otimes_\Lambda r}$. Let
$\bar{v}=\bar{v}_1\otimes...\otimes\bar{v}_r$ and $g\in \GL(V)$,
then
$
g.\bar{v}=g.\bar{v}_1\otimes...\otimes g.\bar{v}_r.
$
We denote the corresponding representation by $\rho^*_r$.

Denote by $\gamma: T^r(V^*)\otimes_{\Lambda} T^r(V)\longrightarrow\Lambda$ the natural
pairing, that is, the $\Lambda$-linear function defined by
\begin{eqnarray}\label{eq:contr-tensor}
\gamma: \bar{v}_1\otimes...\otimes\bar{v}_r\otimes w_1\otimes...\otimes w_r\mapsto
(-1)^{J(v, w)} \bar{v}_1(w_1)\bar{v}_2(w_2)...\bar{v}_r(w_r),
\end{eqnarray}
where $J(v, w)$ is defined as follows. If the elements
$\bar{v}_i$ and $w_i$ are all homogeneous, let
$d_\mu=[w_\mu]([\bar{v}_{\mu+1}]+[\bar{v}_{\mu+2}]+\dots+[\bar{v}_r])$ with
$d_r=0$. Then $J(v, w)=\sum_{\mu=1}^r d_\mu$; extend this definition by
$\La$-multilinearity.

Theorem \ref{thm:BR} is equivalent to the following result, which is
another reformulation of \cite[Theorem 1.1]{S1}.
\begin{theorem}\label{thm:fft-gl}
Let $\Phi: T^r(V^*)\otimes_{\Lambda} T^s(V)\longrightarrow\Lambda$
be a $\Lambda$-linear
$\GL(V)$-invariant function. Then $\Phi\ne 0$ only when $r=s$,
and in this case, $\Phi$ belongs to
the $\Lambda$-span of functions $\gamma_\pi = \gamma\circ(\id\otimes\pi)$,
where $\pi\in \Sym_r$.
\end{theorem}
\begin{proof}  We give a sketch of the easy proof here. It is clear that there are no nonzero functions if $r\ne s$.
When $r=s$, we have the following $\GL(V)$-module isomorphisms
\[
(T^r(V^*)\otimes_{\Lambda} T^r(V))^* \stackrel{\sim}{\longrightarrow}
T^r(V)\otimes_{\Lambda} T^r(V^*) \stackrel{\sim}{\longrightarrow}  \End_\Lambda (T^r(V)).
\]
Thus Theorem \ref{thm:BR} identifies the set of $\Lambda$-linear $\GL(V)$-invariant functions
with $\varpi_r(\Lambda\Sym_r)$. The theorem now follows from
the obvious fact that the function $\gamma$ defined by \eqref{eq:contr-tensor} is $\GL(V)$-invariant.
\end{proof}

\subsection{Polynomials and polynomial functions}\label{sect:poly}
Let $M_\C$ be a $\Z_2$-graded $\C$-vector space
of superdimension $\sdim M=(k|l)$, and let $M=M_\C\otimes_\C\Lambda$.
We shall define the graded symmetric algebra on $M$, and show that it may be used
to define the notion of polynomial functions on $M_{\bar0}$.

\begin{definition}\label{def:symalg}
Let $T(M)=\oplus_{r\geq 0}T^r(M)$ where $T^r(M)=M^{\ot_\Lambda r}$, be the tensor algebra on $M$.
This is a graded superalgebra.
Let $I(M)$ be the ideal of $T(M)$ generated by all elements of the form $u\ot v-(-1)^{[u][v]}v\ot u$,
where $u,v\in M$. The graded symmetric algebra (also referred to as the supersymmetric algebra)
 $S(M)=S_\Lambda(M)$ is defined by $S(M)=T(M)/I(M)$.
\end{definition}

Clearly $S(M)$ is an associative graded-commutative superalgebra. We write $S^r(M)$ for the image of
$T^r(M)$ in $S(M)$.

Recall \eqref{eq:sym} that the symmetric group $\Sym_r$ acts on $T^r(M)$; it is easily shown that

\begin{lemma}\label{lem:sympower}
We have $S^r(M)\cong T^r(M)/\Sym_r$.
\end{lemma}
\begin{remark}\label{rem:syminv}
Since $T^r(M)/\Sym_r$ may be identified with $e(r)T^r(M)$, where $e(r)=
(r!)\inv\sum_{\sigma\in\Sym_r}\sigma\in\C \Sym_r$,
it follows that we also have
$$
S^r(M)\cong T^r(M)^{\Sym_r}.
$$
We therefore have the splitting $T^r(M)=e(r)T^r(M)\oplus (1-e(r))T^r(M)\cong S^r(M)\oplus (1-e(r))T^r(M)$.
Thus $S^r(M)$ is canonically identified with both a quotient and subspace of $T^r(M)$.
\end{remark}
Since the action of $\GL(M)$ on $T^r(M)$ commutes with that of $\Sym_r$, the next statement is clear.
\begin{corollary}\label{cor:gl-on-s}
The diagonal action of $\GL(M)$ on $T^r(M)$ descends canonically to an action on $S^r(M)$.
\end{corollary}
Now given the map $\gamma:T^r(M^*)\ot T^r(M)\to \La$ of \eqref{eq:contr-tensor} and
writing $\cF(M,\Lambda)$ for the $\Lambda$-module of functions $:M\to\Lambda$,
we have a $\La$-linear map $F^r:T^r(M^*)\to \cF(M,\Lambda)$ given by $F^r(\ol{\bf v})(v)=
\gamma(\ol{\bf v}\ot v\ot v\ot\dots\ot v)$,
where $\ol{\bf v}\in T^r(M^*)$ and $v\in M$.

Let $\tau_i$ be the involution $\tau$ (cf. \eqref{eq:sym}) acting on the $i,i+1$ components of $T^r(M^*)$.


\begin{lemma}\label{lem:symfn}
\begin{enumerate}
\item We have, for $\ol{\bf v}\in T^r(M^*)$ and $v\in M$,
$$
F^r(\tau_i(\ol{\bf v}))(v)=(-1)^{[v]}F^r(\ol{\bf v})(v).
$$
\item The map $F^r:T^r(M^*)\to \cF(M,\La)$ restricts to a map $F_r:S^r(M^*)\to \cF(M_{\bar0},\La)$. That is, we have
$\Lambda$-linear maps
$$
T^r(M^*)\overset{/\Sym_r}{\lr}S^r(M^*)\overset{F^r}{\lr}\cF(M_{\bar0},\La).
$$
\end{enumerate}
\end{lemma}
\begin{proof}
The first part is a straightforward calculation, which is left to the reader.
Given (1) it follows that the composition of $F^r$ with restriction to
$M_{\bar0}$  is constant on orbits of $\Sym_r$. In view of Lemma
\ref{lem:sympower}, this implies (2). Note that we here abuse notation
by also writing $F^r$ for the induced map on $S^r(M^*)$.
\end{proof}

\begin{definition}\label{def:polyfn}
The image of the map $F^r$ is called the space of polynomial functions of degree $r$ on $M_{\bar0}$.
It is denoted $\cP^r[M_{\bar0}]$. Specifically, a function $f:M_{\bar0}\to\La$ is polynomial of degree $r$
if there is an element ${\ol\phi}\in T^r(V^*)$ such that for
 $v\in M_{\bar0}$, $f(v)=\gamma(\ol\phi\ot v\ot\dots\ot v)$,
where $\gamma$ is the pairing \eqref{eq:contr-tensor}.
\end{definition}

\begin{remark} \begin{enumerate}
\item We shall make use of the fact that
polynomial maps $M_{\bar0}\lr \La$ restrict to polynomial maps $M_\C\ot\La(N)\lr\La(N)$
for each positive integer $N$,
where $M_\C\ot\La(N)$ and $\La(N)$ are considered as finite dimensional affine spaces over $\C$.
However the ring structure of
$\cP^r[M_{\bar0}]$ is independent of this interpretation of polynomial functions.
\item In the interpretation of elements of $\cP^r[M_{\bar0}]$ as polynomial maps on
finite dimensional affine spaces, each
function $f\in\cP^r[M_{\bar0}]$ has degree $r$ in the sense that $f(\alpha x)
=\alpha^rf(x)$ for $\alpha\in\La_{\bar0}$ and $x\in M_{\bar0}$.
\item In view of (2) we write
$\cP[M_{\bar0}]=\oplus_{r=0}^\infty\cP^r[M_{\bar0}]=\im(\oplus_{r=0}^\infty F^r):T(M^*)\lr\cF(M_{\bar0},\La)$,
and note that the sum is indeed direct.
\item With the same abuse of notation as above, when convenient, we shall think of $\cP[M_{\bar0}]$ as
$\im(F^{\bullet}:=\oplus_{r=0}^\infty F^r):S(M^*)\lr\cF(M_{\bar0},\La)$, where $S(M^*)=S_\La(M^*)$ is
the graded symmetric algebra
defined in Definition \ref{def:symalg}.
\end{enumerate}
\end{remark}

\begin{proposition}\label{prop:poly-polyfn}
The map $F^\bullet:S(M^*)\lr\cP[M_{\bar0}]$ is an isomorphism of graded-commutative graded $\La$-algebras
(which will also be referred to as associative superalgebras).
If $\la\in\La_{\bar0}$ and $f\in\cP^r[M_{\bar0}]$ then for $v\in M_{\bar0}$, $f(\la v)=\la^rf(v)$.
\end{proposition}
\begin{proof}
It is clear that $F^\bullet$ is a surjective homomorphism. Hence it remains only to show that $F^\bullet$ is injective,
i.e. that $F^r$ is injective for each $r$.

 Let $b_1,\dots,b_{k+l}$ be a homogeneous basis of $M$, and let $\phi_1,\dots,\phi_{k+l}$ be the dual basis of $M^*$,
so that $\phi_i(b_j)=\delta_{ij}$.

 For $i=1,2,\dots,k+l$ write $[i]=[b_i]=[\phi_i]$, so that
$[i]=\bar0$ if $1\leq i\leq k$ and $[i]=\bar1$ if $k+1\leq i\leq k+l$. Then $T^r(M^*)$ has $\La$-basis
$\{\phi_{i_1}\ot\dots\ot\phi_{i_{k+l}}\mid 1\leq i_j\leq k+l\}$, which is homogeneous, with
$[\phi_{i_1}\ot\dots\ot\phi_{i_{k+l}}]=\sum_j[i_j]$.
Now up to sign, the image of $\phi_{i_1}\ot\dots\ot\phi_{i_{k+l}}$ in $S^r(M)$ depends only on the
number $m_i$ of occurrences of $\phi_i$ among the $\phi_{i_j}$, $j=1,\dots,r$. Note also that
if $[i]=\bar1$, then $\phi_i\ot\phi_i\mapsto 0\in S(M^*)$. It follows that $S^r(M^*)$ has $\Lambda$-basis
$\{\phi_1^{m_1}\phi_2^{m_2}\dots\phi_{k+l}^{m_{k+l}}\mid m_1+\dots+m_{k+l}=r\}$,
where $\phi_1^{m_1}\phi_2^{m_2}\dots\phi_{k+l}^{m_{k+l}}=
\eta_r(\phi_1\ot\dots\ot\phi_1\ot\phi_2\ot\dots\ot\phi_2\ot\phi_3\ot\dots\ot\phi_{k+l})$, with $\phi_i$
occurring $m_i$ times,
and $m_i\leq 1$ if $k+1\leq i\leq k+l$. Here $\eta_r$ is the natural map $:T^r(M^*)\lr S^r(M^*)$.
This could equally be expressed structurally by the equation
$S(M^*)\cong S(M_{\bar0}^*)\ot_\La\wedge(M_{\bar1}^*)$.

The elements of $S^r(M^*)$ may therefore be uniquely represented by expressions

\be\label{eq:poly}
f=\sum_{m_1+m_2+\dots+m_{k+l}=r}\la_{m_1,m_2,\dots,m_{k+l}}
\phi_1^{m_1}\phi_2^{m_2}\dots\phi_{k+l}^{m_{k+l}},
\ee
 where
$\la_{m_1,m_2,\dots,m_{k+l}}\in\La$, $m_i\in\Z_{\geq 0}$, and $m_i=0$ or $1$ if $k+1\leq i\leq k+l$.
Conversely each such expression defines an element of $S^r(M^*)$.

The value of $F^r(f)$ ($f$ as above) on $v=\sum_i b_i\la_i\in M_{\bar0}$ is given by
\be\label{eq:frv}
F^r(f)(v)=\sum_{m_1+m_2+\dots+m_{k+l}=r}\la_{m_1,m_2,\dots,m_{k+l}}\la_1^{m_1}
\la_2^{m_2}\dots\la_{k+l}^{m_{k+l}}.
\ee
Suppose that $f\in\ker(F^r)$, i.e. that
 $F^r(f)(v)=0$ for all $v\in M_{\bar0}$. We wish to show that $\la_{m_1,m_2,\dots,m_{k+l}}=0$
for all $m_1,m_2,\dots,m_{k+l}$.
 Fix $m_{k+1},\dots,m_{k+l}$ and write $s=r-\sum_{i=k+1}^{k+l}m_i$.
Now let $N$ be any integer such that for all $m_1,\dots,m_{k+l}$, $\la_{m_1,\dots,m_{k+l}}\in\La(N)$.
In \eqref{eq:frv}, take $\la_j=e_{N+j}\in\La(N+j)$ for $j=k+1,\dots,k+l$, (note that
since $v\in M_{\bar0}$, $\la_j\in\La_{\bar1}$ for $j>k$). Then for
 $\la_i\in\La(N)$ for $i=1,\dots,k$, putting $F^r(f)=0$ in \eqref{eq:frv}
(and maintaining the chosen values of $m_{k+1},\dots,m_{k+l}$) yields that
$$
\sum_{m_1+\dots+m_k=s}\la_{m_1,m_2,\dots,m_{k+l}}\la_1^{m_1}\la_2^{m_2}\dots\la_{k}^{m_{k}}=0
$$
for all $\la_1,\dots,\la_k\in\La(N)_{\bar0}$. This implies $\la_{m_1,m_2,\dots,m_{k+l}}=0$ for all sequences
$m_1,\dots,m_k$ such that $m_1+\dots+m_k=s$  (it suffices to take $\la_i\in\C$ for this statement).
Since this is true for all sequences $m_{k+1},\dots, m_{k+l}$, it follows that $f=0$.

The second statement is evident from the above proof.
\end{proof}
\begin{remark}\label{rem:density}
\begin{itemize}
\item Note that the above proof shows that if a polynomial function $f$ vanishes on the set
$\sum_i \la_i m_i$, where $m_i$ runs over a homogeneous basis of $M$ and $\la_i\in\C$ if $[m_i]=\bar0$,
then $f=0$. We therefore say that this set (and any other set with this property) is {\em Zariski dense} in $M_{\bar0}$.
\item We shall refer to the variables $\phi_i$ in \eqref{eq:poly} as either `polynomial variables' (if $1\leq i\leq k$)
or `Grassmann variables' (if $k+1\leq i\leq k+l$). A polynomial variable may have an arbitrary integer exponent, whereas
a Grassmann variable has exponent either $0$ or $1$. Polynomial variables will also be referred to
as `even variables'.
\end{itemize}
\end{remark}
We shall require the following observation later.

\begin{lemma}\label{lem:polyscal}
Let $M=M_\C\ot_\C\La$ be as above, and suppose $M$ is an orthosymplectic superspace with form $(-,-)_M$.
Let $W=M\oplus M$; this is also an orthosymplectic superspace in the obvious way.
Then the function $q:v\oplus w\mapsto (v,w)_M$ ($v,w\in M$) lies in $\cP^2[W_{\bar0}]$.
\end{lemma}
\begin{proof}
Let $b_1,\dots b_{k+l}$ be a homogeneous basis of $M$, regarded as the first summand of $W$,
 and let $b_1',\dots b_{k+l}'$ be its graded dual basis,
so that $(b_i,b_j')_M=(-1)^{[i]}\delta_{ij}$, where $[i]=[b_i]=[b_i']$. Similarly, let $(c_i),(c_i')$ be dual bases
of the second summand of $W$. Let $v,w\in M_{\bar0}$, and write
$v=\sum_ib_i\la_i$, and $w=\sum_ic_i'\mu_i$, where for each $i$, $[\la_i]=[\mu_i]=[i]$ since $v,w\in M_{\bar0}$.
Then $(v,w)_M=\sum_i\la_i\mu_i$.

For any element $x\in M$ define $\phi_x\in M^*$ by $\phi_w(v)=(w,v)_M$. Then
$$
\begin{aligned}
F^2(\phi_{b_i}\ot\phi_{c_i'})(v\oplus w)=&\gamma(\phi_{b_i}\ot\phi_{c_i'}\ot (v\oplus w)\ot (v\oplus w))=
(-1)^{[v\oplus w][i]}(b_i,v)(c_i',w)\\
=&(-1)^{[i]}\mu_i\la_i=\la_i\mu_i.\\
\end{aligned}
$$
It follows that $q=F^2(\sum_i \phi_{b_i}\ot\phi_{c_i'})\in\cP^2[W_{\bar0}]$.
\end{proof}




\subsection{Polynomial functions on endomorphism algebras}\label{ssec:polyfn}
We now take $V_\C$ to be a $\Z_2$-graded vector space and apply the considerations of the last subsection to the
superspace $M_\C=\End_\C(V_\C)$.

Consider the following left actions of $\GL(V)$ on $\cP^r[\End_\Lambda(V)_{\bar0}]$,
\[
\begin{aligned}
&L: \GL(V)\times\cP^r[\End_\Lambda(V)_{\bar0}]\longrightarrow\cP^r[\End_\Lambda(V)_{\bar0}],
\quad L: (g, f)\mapsto L_g.f, \\
&R: \GL(V)\times\cP^r[\End_\Lambda(V)_{\bar0}]\longrightarrow\cP^r[\End_\Lambda(V)_{\bar0}],
\quad R: (g, f)\mapsto R_g.f,
\end{aligned}
\]
where
\[
(L_g.f)(X)= f(g^{-1} X), \quad (R_g.f)(X)=f(Xg), \quad \forall X\in\End_\Lambda(V)_{\bar0}.
\]
Note that these actions evidently commute with each other, i.e. for $g\in\GL(V)$, $L_gR_g=R_gL_g$.

Observe that the map $\gamma$ of \eqref{eq:contr-tensor} defines a canonical perfect pairing
which canonically identifies $T^r(V^*)$ with $T^r(V)^*$. We shall use this identification below
without comment.

For any linear form $L\in T^r(V^*)$, define the function
$\Psi_L\in \cF(\End_\Lambda(V)_{\bar0},\La)\ot_\La \cF(T^r(V),\La)
(\cong \cF(\End_\Lambda(V)_{\bar0}\times T^r(V),\La)$ by
\be\label{eq:defpsiL}
\Psi_L(X, w)=L(X.w),
\ee
for all $X\in \End_\Lambda(V)_{\bar0}$ and $w\in T^r(V)$,
where,  for $w= w_1\otimes...\otimes  w_r$, $X.w=X w_1\otimes...\otimes X w_r$.

\begin{lemma}\label{lem:psipoly}
For any $L\in T^r(V^*)$, the function $\Psi_L\in
\cP^r[\End_\Lambda(V)_{\bar0}]\otimes_\Lambda T^r(V)^*$.
That is, $\Psi_L(X,w)$ is linear in $w$, and polynomial of degree $r$ in $X$.
\end{lemma}
\begin{proof}
We need to show that for fixed $X\in  \End_\Lambda(V)_{\bar0}$, the map $\Psi_L(X,-):w\mapsto \Psi_L(X,w)$
is linear, and that for each $w\in T^r(V)$, the map
$\Psi_L(-,w):X\mapsto \Psi_L(X,w)$ belongs to $\cP^r[\End_\Lambda(V)_{\bar0}]$.
The first statement is clear from the definition \eqref{eq:defpsiL}.

To prove that $\Psi_L(-,w)\in\cP^r[\End_\Lambda(V)_{\bar0}]$, note that since the space of polynomial
functions of degree $r$ is closed under $\La$-linear combinations, it suffices to take $w= w_1\otimes...\otimes  w_r$,
where $w_i\in V$. Define $\phi\in T^r(\End_\Lambda(V)^*)=T^r(\End_\Lambda(V))^*$ by
$$
\phi(X_1\ot X_2\ot\dots\ot X_r):=L(X_1w_1\ot X_2w_2\ot\dots\ot X_rw_r).
$$
It is now easily verified that in the notation of Lemma \ref{lem:symfn}(2),
$\Psi_L(-,w)=F^r(\phi)\in\cP^r[\End_\Lambda(V)_{\bar0}]$.
\end{proof}
Now define the following $\GL(V)$-actions on $\cP^r[\End_\Lambda(V)_{\bar0}]\otimes_\Lambda T^r(V^*)$:
$$
\hat{L}_g=L_g\otimes \id\text{ and }\hat{R}_g=R_g\otimes\rho_r^*(g)\;\;\;\forall g\in \GL(V).
$$
It follows since $L$ and $R$ commute, that $\hat{L}$ commutes with $\hat{R}$. Thus
the invariant subspace
\begin{eqnarray}\label{eq:Gamma}
\Gamma_r:=\Big(\cP^r[\End_\Lambda(V)_{\bar0}]\otimes_\Lambda T^r(V^*)\Big)^{\hat{R}(\GL(V))}
\end{eqnarray}
is a left $\GL(V)$-module under the action $\hat{L}$.
\begin{remark}
This construction is analogous  to that of \cite[Theorem 4.2]{SZ}.
\end{remark}

The next result is an important link in our development.

\begin{proposition}\label{prop:Borel-Weil}
Let $\Psi$ be the $\Lambda$-linear map
$\Psi: T^r(V^*)\longrightarrow \cP^r[\End_\Lambda(V)_{\bar0}]\otimes_\Lambda T^r(V^*)$
given by $\Psi (L)= \Psi_L$.
Then the image of $\Psi$ is contained in $\Gamma_r$, and $\Psi:T^r(V^*)\lr\Gamma_r$
is an isomorphism of $\GL(V)$-modules.
\end{proposition}
\begin{proof}
We continue to identify $\cP^r[\End_\Lambda(V)_{\bar0}]\otimes_\Lambda T^r(V^*)$ with a subspace
of the $\Lambda$-module $\cF(\End_\Lambda(V)_{\bar0}\times T^r(V),\Lambda)$ of functions from
$\End_\Lambda(V)_{\bar0}\times T^r(V)$ to $\Lambda$.

First observe that $\Psi$ is $\GL(V)$-equivariant. For if $g\in\GL(V)$, then for $L\in T^r(V^*)$
and $(X,w)\in \End_\Lambda(V)_{\bar0}\times T^r(V)$, we have
$\Psi_{gL}(X,w)=L(g\inv X.w)=\Psi_L(g\inv X,w)=\hat L_g(\Psi_L)(X,w)$, whence $\Psi_{gL}=\hat L_g\Psi_L$.

We next show that for any $L\in T^r(V^*)$, $\Psi_L\in\Gamma_r$. Observe that if $g\in\GL(V)$,
then for $w=w_1\ot\dots\ot w_r\in T^r(V)$, $\hat R_g.\Psi_L(X,w)=\Psi_L(Xg,g\inv w)
=L(Xgg\inv w_1\ot\dots\ot Xgg\inv w_r)=L(Xw_1\ot\dots\ot Xw_r)=\Psi_L(X,w)$, whence
$\Psi_L\in\Gamma_r$.

Now suppose $L\in\ker\Psi$. Then $\Psi_L(X,w)=0$ for all $(X,w)\in \End_\Lambda(V)_{\bar0}\times T^r(V)$.
In particular $\Psi_L(\id_V,w)=L(w)=0$ for all $w\in T^r(V)$. It follows that $\Psi$ is injective, and hence that
it now suffices to prove that $\Psi$ is surjective.

With this in mind, take an arbitrary element
$\Phi\in \Gamma_r\subseteq\cP^r[\End_\Lambda(V)_{\bar0}]\otimes_\Lambda T^r(V^*)$.
Since $\Phi\in\Gamma_r$ is invariant under $\hat R(\GL(V))$, we have, for each element $g\in\GL(V)$,
$\hat R_g\Phi=\Phi$. But for $(X,w)\in \End_\Lambda(V)_{\bar0}]\times T^r(V)$,
$\hat R_g\Phi(X,w)=\Phi(Xg,g\inv w)$.

Suppose $X\in\GL(V)$. Then taking $g=X\inv$ in the last equation, we obtain
$$
\Phi(X,w)=\Phi(\id_V, Xw)=\Psi_L(X,w),
$$
where $L\in T^r(V^*)=T^r(V)^*$ is defined by $L(w)=\Phi(\id_V,w)$.
It follows that for this $L$, the functions $\Phi$ and $\Psi_L$ coincide on
$\GL(V)\times T^r(V)$. We show finally that this implies that $\Phi=\Psi_L$
on the whole of $\End_\La(V)_{\bar0}\times T^r(V)$, which will show that
$\Phi\in\im\Psi$, so that $\Psi$ is surjective.

Fix $w\in T^r(V)$. Then $\Phi(-,w)$ and $\Psi_L(-,w)$ are two elements of $\cP^r(\End_\La(V)_{\bar0}$
which coincide on $\GL(V)$. We shall show that this implies that they coincide on the whole of $\End_\La(V)_{\bar0}$.
Let $\La(N)\subset\La$ be the finite dimensional Grassmann algebra of dimension $2^N$.
Then $\End_\La(V)_{\bar0}\cong (\End_\C(V_\C)\otimes_\C\La)_{\bar0}$, and each element of
$\End_\La(V)_{\bar0}$ lies in the finite dimensional $\C$-subspace $ E(N):=
(\End_\C(V_\C)\otimes_\C\La(N))_{\bar0}$
of $\End_\La(V)_{\bar0}$ for some $N$. It therefore suffices to show that
$\Phi(-,w)=\Psi_L(-,w)$ on $E(N)$ for each integer $N>0$.

Now both $E(N)$ and $\Lambda(N)$ are finite dimensional vector spaces
over $\C$, and it is easily seen that the restriction to $E(N)$ of any element $\Xi$ of $\cP^r[\End_\La(V)_{\bar0}]$
is a polynomial function $E(N)\lr\La(N)$ in the usual sense of affine $\C$-spaces. Let $\Xi=\Phi(-,w)-\Psi_L(-,w)$.
Then  $\Xi=0$ on
$\GL(V)\cap E(N)$. But (cf. proof of Lemma \ref{lem:exp}(2)) an element of $E(N)$ is invertible, i.e.
lies in $\GL(V)$, precisely when its degree zero component is an invertible element of
$\GL(V_\C)$. Since this is an open condition, $\Xi=0$ on a Zariski dense subset of $E(N)$,
whence $\Xi=0$ on the whole of $E(N)$ for any integer $N>0$. It follows that $\Xi=0$, i.e. $\Phi(-,w)=\Psi_L(-,w)$
on the whole of $\End_\La(V)_{\bar0}$, and hence that $\Phi\in\im\Psi$.
\end{proof}
\subsection{Some $\GL(V)$-module isomorphisms}\label{sect:gl-iso}
We point out in this section some identifications of $\Lambda$-modules
analogous to the standard ones in linear algebra. Let $V$ be an orthosympletic superspace.

For any $A\in \End_\Lambda(V)$, $A^\dag$ will denote the unique element
of $\End_\Lambda(V)$ satisfying
\[
(A v, w)= (-1)^{[A][v]}(v, A^\dag w), \quad \forall v, w\in V.
\]
We refer to
$A^\dag$ as the adjoint of $A$ with respect to the form $(-, - )$.
It is easy to see that the map $A\mapsto A^\dag$ is
$\Lambda$-linear, i.e. that for $A\in \End_\Lambda(V)$ and $\lambda\in\Lambda$,
$(A\lambda)^\dag=A^\dag\lambda$, and preserves parity, i.e. $[A^\dag]=[A]$ for all $A\in \End_\Lambda(V)$.
Moreover, for $\la\in\La$, $(\la\id_V)^\dag=\la\id_V$, and for all $A, B\in \End_\Lambda(V)$, we have the following
three relations.
\begin{eqnarray}\label{eq:involution}
\begin{aligned}
&(A B)^\dag= (-1)^{[A][B]}B^\dag A^\dag, \quad (A^\dag)^\dag=A, \\
&\text{and }(A^\dag)^{-1}=(-1)^{[A]}(A^{-1})^\dag \quad \text{if $A$ is invertible}.
\end{aligned}
\end{eqnarray}

\begin{definition} Write $\End_\Lambda(V)^{\pm} = \{A\in \End_\Lambda(V) \mid A^\dag=\pm A\}$.
Recalling that $\bE=\End_\Lambda(V)_{\bar0}$; we also write $\bE^\pm=\End_\Lambda(V)_{\bar0}^{\pm}$.
Elements of $\End_\Lambda(V)^+$ are said to be self  adjoint.
\end{definition}
The second equation in \eqref{eq:involution}, as well as the observation that for any $X\in\End_\La(V)$, we have
$X=\frac{X+X^\dag}{2}+\frac{X-X^\dag}{2}$,  implies that
$\End_\Lambda(V)$ decomposes as a direct sum of free $\Lambda$-modules as follows.
\begin{eqnarray}\label{eq:decomp-end}
\End_\Lambda(V)=\End_\Lambda(V)^+\oplus\End_\Lambda(V)^-.
\end{eqnarray}

By taking the intersection with $\bE=\End_\La(V)_{\bar0}$ of the above decomposition,
we also obtain
\be\label{eq:decomp}
\bE=\bE^+\oplus \bE^-,
\ee
where $\bE^\pm=\End_\La(V)^\pm\cap \bE$.

\begin{remark}
 Since $\OSp(V)\subset \bE=\End(V)_{\bar0}$, $\OSp(V)=\{g\in \GL(V)\mid g^\dag g=\id\}$.

\end{remark}

For any $A\in\GL(V)$, we define $\hat A := (A^\dag)^{-1}$.  Then
$\widehat{AB}= \hat A \hat B$ for all $A, B\in\GL(V)$, and $g\mapsto \hat g$ defines an involutory automorphism
of $\GL(V)$.  This enables us to
define a twisted $\GL(V)$-action on $V$ as follows.
\begin{eqnarray}\label{eq:twist}
\GL(V)\times V \longrightarrow V, \quad (g, v) \mapsto \hat{g}v.
\end{eqnarray}
\begin{lemma}\label{lem:isodual}
Let $\phi: V\stackrel{\sim}{\longrightarrow} V^*$, $v\mapsto \phi_v$,  be the
map defined by $\phi_v(w) =(v, w)$ for all $w\in V$.
Then $\phi$ is a $\La$-module isomorphism from $V$ to $V^*$, which respects the $\GL(V)$-action,
where $\GL(V)$ acts on $V$ with the twisted action \eqref{eq:twist} and on
$V^*$ with the standard action. That is, for any $g\in\GL(V)$ and $v\in V$, $ \phi_{\hat{g} v}=g.\phi_v $.
\end{lemma}
\begin{proof} Clearly $\phi_v\in V^*$ for any $v\in V$. Moreover for $\la\in\La$,
 $\phi_{\la v}(x) =(\la v,x)=\la(v,x)=\la\phi_v(x)$, so that the map $v\mapsto \phi_v$ is a $\La$-module
homomorphism by \eqref{eq:lambda-module}, and is an isomorphism by the non-degeneracy of $(-,-)$
The standard $\GL(V)$-action on $V^*$ is defined by
$(g\bar{v})(w)=\bar{v}(g^{-1}w)$ for any $\bar{v}\in V^*$, $g\in\GL(V)$
and $w\in V$. Thus $g\phi_v(w)=\phi_v(g\inv w)=(v,g\inv w)=((g\inv)^\dag v,w)=(\hat g v,w)$
for all $g\in G$ and $v,w\in V$. Hence
$g.\phi_v = \phi_{\hat{g} v}$ for all $v\in V$.
\end{proof}

Next observe that there is a $\GL(V)$-action on $\End_\Lambda(V)$ defined by
\begin{eqnarray}\label{eq:gl-action}
\Upsilon: \GL(V) \times \End_\Lambda(V) \longrightarrow
\End_\Lambda(V), \quad (g, A)\mapsto \Upsilon_g(A)=\hat{g}A g^{-1}.
\end{eqnarray}
Clearly both $\bE=\End_\Lambda(V)_{\bar0}$ and $\End_\Lambda(V)_{\bar1}$ are $\GL(V)$-stable
$\Lambda_{\bar 0}$-submodules of $\End_\Lambda(V)$ under this action.
Further, since $ \Upsilon_g(A)^\dag= \Upsilon_g(A^\dag)$ for all $A\in \bE$, both $E^+$
and $E^-$ are stable under the action of $\GL(V)$, as are $\bE^{\pm}$.

Consider the even $\Lambda$-linear map
\begin{eqnarray}\label{eq:zeta}
\zeta^*: V^*\otimes_\Lambda V^* \longrightarrow \End_\Lambda(V)
\end{eqnarray}
defined by $\phi_v\otimes \phi_w\mapsto  \phi_{v,w}$, where $\phi_{v,w}(x)=v(w, x)$
for all $v,w,x\in V$. This is clearly a $\Lambda$-module isomorphism. However more is true.
%

\begin{lemma}
The map $\zeta^*$ is a $\GL(V)$-module isomorphism, with $\GL(V)$ acting on
$\End_\Lambda(V)$ as in \eqref{eq:gl-action}. Furthermore, $\zeta^*$
restricts to $\GL(V)$-module isomorphisms
$(V^*\otimes_\Lambda V^*)_{\bar\alpha} \stackrel{\sim}{\longrightarrow}
\End_\Lambda(V)_{\bar\alpha}$ for ${\bar\alpha}=\bar0, \bar1$.
\end{lemma}
\begin{proof} The second statement follows from the first and the fact that $\zeta^*$ is even.
Since $\zeta^*$ is evidently a $\Lambda$-module isomorphism, we only need to prove its $\GL(V)$-equivariance.
Given the element $X=\phi_v\otimes \phi_w$ of $V^*\otimes_\Lambda V^*$, we have
 $\zeta^*(X)=\phi_{v,w}$, and for $x\in V$,
\[
\begin{aligned}
\zeta^*(g.X)(x)
&=\zeta^*(g.\phi_v\otimes g.\phi_w)=\zeta^*(\phi_{\hat{g}v}\otimes \phi_{\hat{g}w})(x)\\
&=\phi_{\phi_{\hat{g}v},\phi_{\hat{g}w}}(x)=\hat{g}v({\hat{g}w},x)\\
&=\hat{g}v({w},\hat{g}^\dag x)=\hat{g}v({w},g\inv x)\\
&=\hat{g}\phi_{v,w}g\inv(x)=\Upsilon_g(\zeta^*(X))(x).
\end{aligned}\]
\end{proof}

Now every element of $V^*\otimes_\Lambda V^*$ (resp. $(V^*\otimes_\Lambda V^*)_{\bar0}$) can be expressed as
a $\Lambda$ (resp. $\Lambda_{\bar0}$)-linear combination
of elements of the form $\phi_v\otimes\phi_w$, for homogeneous elements $v,w\in V$,
where  in the case of $(V^*\otimes_\Lambda V^*)_{\bar0}$, $v, w$ have the same parity.
Recall that we have the $\Lambda_{}$-involution on $V^*\otimes_\Lambda V^*$
defined by $\tau(\phi_v\otimes\phi_w)=-1^{[v][w]}\phi_w\otimes\phi_v$,
which restricts to a $\Lambda_{\bar 0}$-involution on $(V^*\otimes_\Lambda V^*)_{\bar0}$.
 Let
\[
\begin{aligned}\label{eq:tensordec}
S^2(V^*)&=\{w \in (V^*\otimes_\Lambda V^*)\mid \tau(w)=w\},\\
\wedge^2(V^*)&=\{w \in (V^*\otimes_\Lambda V^*)\mid \tau(w)=-w\}.
\end{aligned}
\]
Then $V^*\ot V^*=S^2(V^*)\oplus \wedge^2(V^*)$, and $(V^*\ot V^*)_{\bar0}
=S^2(V^*)_{\bar0}\oplus \wedge^2(V^*)_{\bar0}$.

Note that $S^2(V)$ is isomorphic to the degree $2$ part of the symmetric algebra
defined in Definition \ref{def:symalg} above,
by virtue of the decomposition \eqref{eq:tensordec}. Thus this notation is consistent with the notation above.


\begin{lemma}\label{lem:iso}
\begin{enumerate}
\item We have, for all homogeneous $v,w\in V$,
$$
\phi_{v,w}^\dag=(-1)^{[v][w]}\phi_{w,v}.
$$
\item For all $X\in V^*\ot V^*$, we have
$$
\zeta^*(\tau(X))=\zeta^*(X)^\dag.
$$
That is, $\zeta^*\circ\tau=^\dag\circ\zeta^*$.
\item We have
$\zeta^*(S^2(V^*)_{\bar0})=\End_\Lambda(V)_{\bar0}^+$ and $
\zeta^*(\wedge^2(V^*)_{\bar0})=\End_\Lambda(V)_{\bar0}^-.$
\end{enumerate}
\end{lemma}
\begin{proof}
For all $x, y\in  V$, we have
\[
\begin{aligned}
(\phi_{v,w}(x), y) &= (-1)^{[v]([w]+[x])}(w, x)(v, y) \\
&= (-1)^{[v][w]+[v][x]+[w][x]}(x, w)(v, y)\\
&=(-1)^{([v]+[w])[x]}(x,(\phi_{v,w})^\dag(y)).\\
\end{aligned}
\]
Hence $(x,(\phi_{v,w})^\dag(y))=(-1)^{[v][w]}(x, w)(v, y)$. But
$(x,\phi_{w,v}(y))=(x,w(v,y))=(x,w)(v,y)$, and (1) follows.

Part (2) for $X=\phi_v\ot \phi_w$ is an easy consequence of (1), given that
$\phi_{v,w}=\zeta^*(\phi_v\ot \phi_w)$, etc. The general form
of (2) follows since the statement is clearly $\Lambda$-linear in $X$, and the elements $\phi_v\ot \phi_w$
with $v,w$ homogeneous elements of $V$, span $V^*\ot V^*$ over $\Lambda$.

Part (3) is an easy consequence of part (2) and the decomposition \eqref{eq:decomp}
of $\End_\Lambda(V)_{\bar0}$.
\end{proof}

Let $\Gamma_r^{\OSp(V)}$ be the subspace of the $\GL(V)$-module $\Gamma_r$ (see \eqref{eq:Gamma})
consisting of $\OSp(V)$-invariants.  Then
\[
\Gamma_r^{\OSp(V)}:=\Big(\cP^r[\End_\Lambda(V)_{\bar0}]^{\hat{L}(\OSp(V))}\otimes_\Lambda
T^r(V^*)\Big)^{\hat{R}(\GL(V))}.
\]
Similarly denote by $T^r(V^*)^{\OSp(V)}$ the subspace of $T^r(V^*)$ consisting of
$\OSp(V)$-invariants. The following result is an immediate consequence of Proposition \ref{prop:Borel-Weil}.
\begin{proposition}\label{prop:BW-OSp}
The map $L\mapsto \Psi_L$ of Pproposition \ref{prop:Borel-Weil} defines a $\Lambda$-module isomorphism
\[
T^r(V^*)^{\OSp(V)}\stackrel{\sim}\longrightarrow\Gamma_r^{\OSp(V)}.
\]
\end{proposition}

Note that $-\id$ belongs to $\OSp(V)$, and $(-\id)(w)=(-1)^r w$ for all $w\in T^r(V^*)$.
Hence $T^r(V^*)^{\OSp(V)}=0$ if $r$ is odd. Thus $\Gamma_r^{\OSp(V)}=0$ by the proposition.

\section{Invariant theory for the orthosymplectic supergroup}\label{sect:inv-osp}

We retain the notation of Sections \ref{sect:gl-iso} and \ref{sect:poly}.
\subsection{Statement of the key lemma}

The orthosymplectic supergroup $\OSp(V)$ acts on the superalgebra of polynomial functions
$\cP[\End_\Lambda(V)_{\bar0}]$ in the
following way. For any $f\in \cP[\End_\Lambda(V)_{\bar0}]$ and $g\in \OSp(V)$,
$(g\cdot f)(A)=f(g^{-1} A)$ for all $A\in\End_\Lambda(V)_{\bar0}$,
where $g^{-1} A$ is the composition of $g^{-1}$ and $A$.
This action respects the associative
superalgebra structure of $\cP[\End_\Lambda(V)_{\bar0}]$ (cf. Proposition \ref{prop:poly-polyfn}).

\begin{lemma}\label{lem:omega}
Let $\omega:\bE\to \bE$, where $\bE=\End_\Lambda(V)_{\bar0}$ be defined by
$\omega(X)=X^\dag X$. Then $\omega(\bE)\subseteq \bE^+$, and if $\alpha\in\cP^r[\bE^+]$
then $\alpha\circ\omega\in\cP^{2r}[\bE]$.
\end{lemma}
\begin{proof}
Since $(X^\dag X)^\dag=X^\dag X$, the first statement is evident.

To prove the second, let $b_1,\dots,b_{k+l}$ be a homogeneous basis of $V$, and
in the notation of Lemma \ref{lem:isodual} and \eqref{eq:zeta},
let $\phi_{ij}=\phi_{b_i,b_j}=\zeta^*(\phi_{b_i}\ot\phi_{b_j})$ for $1\leq i,j\leq k+l$.
Then $\{\phi_{ij}\mid 1\leq i,j\leq k+l\}$ is a homogeneous basis of $\End_\La(V)$;
writing $[b_i]=[i]$, we have $[\phi_{ij}]=[i]+[j]$.

Each element $X\in \End_\La(V)$ has a unique expression $X=\sum_{i,j}\phi_{ij}\la_{ij}$,
with $\la_{ij}\in\La$, and $X\in \bE=\End_\La(V)_{\bar0}$ if and only if $[\phi_{ij}]=[\la_{ij}]=[i]+[j]$
for all $i,j$. Now $X^\dag X$ may easily be computed using the facts that $\phi_{ij}^\dag=(-1)^{[i][j]}\phi_{ji}$,
and that $\phi_{ji}\phi_{pq}=(b_i,b_p)\phi_{jq}$.

The result is that $\omega(X)=X^\dag X=\sum_{j,q}\phi_{jq}\mu_{jq}$, where
\be\label{eq:mu}
\mu_{jq}=\sum_{i,p=1}^{k+l}(-1)^{[i]+[j][q]+[j][p]+[p][q]}\la_{ij}(b_i,b_p)\la_{pq}.
\ee

Now any element $\alpha\in\cP^r(\bE^+)$ is a polynomial of degree $r$ in the $\mu_{jq}$ (see \eqref{eq:frv}).
By \eqref{eq:mu}, the $\mu_{jq}$ are polynomials of degree $2$ in the $\la_{ij}$, whence $\alpha\circ\omega(X)$
is a polynomial of degree $2r$ in the $\la_{ij}$, whence the second statement.
\end{proof}
\begin{lemma}[Key Lemma]\label{lem:key}
Let $G=\OSp(V)$ and $\cP=\cP[\End_\Lambda(V)_{\bar0}]$.
Denote by $\cP^G$ the subalgebra
of $\cP$ consisting of $G$-invariant functions, i.e. functions
$f$ such that for $X\in \bE$, $f(gX)=f(X)$ for all $g\in G$.
Then given any $f\in\cP^G$, there is a polynomial function
$\Phi: \End_\Lambda(V)_{\bar0}^+\longrightarrow\Lambda$
such that $f(A)=\Phi(A^\dag A)$.
\end{lemma}

We postpone the proof to Section \ref{sect:pf-new}. Here we apply it to
prove the first fundamental theorem of invariant theory for
the orthosymplectic supergroup.

\subsection{First fundamental theorem of invariant theory}

For $r=2d$, define the $\Lambda$-linear map
\begin{eqnarray}\label{eq:contraction}
\delta: V^{\otimes_\Lambda 2d}\longrightarrow \Lambda, \quad
v_1\otimes v_2\otimes\dots\otimes v_{2d} \mapsto \stackrel{\rightarrow}{\prod_{1\le i\le d}} (v_{2i-1}, v_{2i}),
\end{eqnarray}
where $\stackrel{\rightarrow}{\prod}_{1\le i\le d}(v_{2i-1}, v_{2i})
=(v_{1}, v_{2})(v_{3}, v_{4})(v_{2d-1}, v_{2d})$.
The symmetric group $\Sym_{2d}$ acts on $V^{\otimes_\Lambda 2d}$ by \eqref{eq:sym},
and corresponding to each $\pi\in\Sym_{2d}$, we have the function
$\delta\circ\pi: V^{\otimes_\Lambda 2d}\longrightarrow \Lambda$.

\begin{theorem}\label{thm:fft-osp}
Let $V$ be as described above, and let $G=\OSp(V)$ be the orthosymplectic supergroup.
Set $W= (V^*)^{\otimes_\Lambda r}$, and denote by $W^G$ the  $\Lambda$-submodule of $G$-invariants
of $W$.
\begin{enumerate}
\item If $r$ is odd, then $W^G=0$.
\item If $r=2d$ is even, then $W^G$ is spanned over $\Lambda$ by $\{\delta\circ\pi \mid \pi\in\Sym_{2d}\}$.
\end{enumerate}
\end{theorem}
\begin{proof}
We have already observed that $W^G= 0$ if $r$ is odd,  and hence it remains only to prove (2).
We will do this by using the Key Lemma to reduce the statement to
the first fundamental theorem of invariant theory for the general linear supergroup,
Theorem \ref{thm:fft-gl} above.

Let $\cP=\cP[\bE^+]$ be the $\Lambda$-superalgebra of polynomial functions $\bE^+ \longrightarrow\Lambda$,
and denote by $\cP^k$ the submodule of homogeneous polynomial functions of degree $k$.
Then $\cP$ is a $\GL(V)$-module with action defined for any $f\in\cP$ and $g\in\GL(V)$ by
\[
(g.f)(A)=f(\Upsilon_{g^{-1}}(A)),  \quad  \forall A\in \bE^+\text{ and }f\in\cP,
\]
where $\Upsilon$ is the action of $\GL(V)$ on $\End_\La(V)$ defined in \eqref{eq:gl-action}.
Thus $\cP\otimes_\Lambda {T^r(V)}^*$ is also
a  $\GL(V)$-module under the diagonal action, which we denote by $\hat\Upsilon^*$.

Let $L\in \big({T^{2d}(V)}^*\big)^G$. By Proposition \ref{prop:BW-OSp},
there exists a unique $\Psi_L\in \Gamma_{2d}^G$ such that $L=\Psi_L(\id)$.
Recall that $\Gamma_{2d}^G=
\left(\cP^{2d}[\bE]^{{\hat L}(G)}\otimes_\Lambda {T^{2d}(V)}^*\right)^{{\hat R}(GL(V))}
\subset \cP^{2d}[\bE]^{{\hat L}(G)}\otimes_\Lambda {T^{2d}}(V)^*$,
where $\cP^{2d}[\bE]$ is the $\Lambda$-module of degree $2d$ homogeneous polynomial functions on $\bE$.
It follows from  Lemma \ref{lem:key} that there exists a function
$F_L\in \cP^d[\bE^+]\otimes_\Lambda {T^r(V)}^*$ such that
\[
F_L(X^\dag X, w)=F_L(\omega(X), w)=\Psi_L(X, w), \quad \forall X\in \bE, \ w\in T^{2d}(V).
\]
Note that $\omega(X) := X^\dag X\in \bE^+$ for all $X\in \bE$, and for
any $g\in \GL(V)$, we have
\[
\omega(Xg^{-1}) = \hat{g}\omega(X) g^{-1}= \Upsilon_g(\omega(X)).
\]
Thus
$
(\hat{R}_g.\Psi_L)(X, w) = \Psi_L(X g, g^{-1}.w) =
F_L(\Upsilon_{g^{-1}}(\omega(X), g^{-1}.w).
$
That is, for $X\in \bE$ and $g\in\GL(V)$,
\[
(\hat{R}_g.\Psi_L)(X, w)  =(\hat\Upsilon^*_g(F_L))(\omega(X), w).
\]
Since $\Psi_L$ is invariant under $\hat R(\GL(V))$ (see Proposition \ref{prop:Borel-Weil}),
we have $\hat\Upsilon^*_g(F_L)=F_L$ for all $g\in\GL(V)$.

It follows from Lemma \ref{lem:iso} that $\bE^+\cong S^2(V^*)_{\bar0}$ as
$\GL(V)$-module over $\Lambda_{\bar 0}$, where $\GL(V)$ acts on $\bE^+$
via $\Upsilon$ and on $S^2(V^*)$ via the inherited action from $T^2(V^*)$
(cf. Corollary \ref{cor:gl-on-s}). Thus $F_L$ corresponds to a
function $H_L$ on $S^2(V^*)_{\bar0}\times T^{2d}(V)$, which is homogeneous of degree
$d$ in $S^2(V^*)_{\bar0}$ and $\Lambda$-linear in $T^{2d}(V)$.
By Definition \ref{def:polyfn} we may regard $H_L$ as a
$\GL(V)$-invariant $\Lambda$-linear function
\[
H_L: S^2(V^*)^{\otimes_\Lambda d}\otimes_\Lambda T^{2d}(V)\longrightarrow\Lambda.
\]
Now using the natural map $T^{2d}(V^*)\lr T^d(S^2(V^*))$, $H_L$ may be lifted to a $\GL(V)$-invariant function
\[
\tilde H_L: T^{2d}(V^*)\otimes_\Lambda T^{2d}(V)\longrightarrow\Lambda.
\]

For elements $\phi_1\ot\dots\ot\phi_{2d}\in T^{2d}(V^*)$, and
$v=v_1\otimes v_2\otimes\dots\otimes v_{2d}\in T^{2d}(V)$,
it follows from Theorem \ref{thm:fft-gl} that $\tilde H_L$ is a $\La$-linear combination of functions of the form
$$
\delta_\pi'(\phi_1\ot\dots\ot\phi_{2d}\ot v_1\otimes v_2\otimes\dots\otimes v_{2d})
=\phi_1(v_{\pi (1)})\dots\phi_{2d}(v_{\pi(2d)})
$$
for some permutation $\pi\in\Sym_{2d}$.

Restricting $\tilde H_L$ to $S^2(V^*)^{\otimes_\Lambda d}\otimes_\Lambda T^{2d}(V)$, it follows that
if $U_1,\dots, U_d\in \bE^+\cong S^2(V^*)_{\bar0}$, then $H_L(U_1\ot\dots\ot U_d\ot v_1\ot\dots\ot v_{2d})$
is a $\La$-linear combination of the functions
$$
\delta_\pi (U_1\ot\dots\ot U_d\ot v_1\ot\dots\ot v_{2d})=(v_{\pi(1)},U_1v_{\pi(2)})\dots(v_{\pi(2d-1)},U_dv_{\pi(2d)}).
$$
This may be seen by a simple calculation when $U_i=\phi_{w_i,w_{i+1}}=\zeta^*(w_i\ot w_{i+1})$, and the general
statement follows by extending linearly.



But for $X\in \bE^+$ and $w=w_1\ot\dots\ot w_{2d}\in T^{2d}(V)$, $F_L(X,w)=H_L(X\ot X\ot\dots\ot X,w)$,
and further,
$L(w)=\Psi_L(\id_V,w)=F_L(\id_V,w)$. It follows that $L( w_1\ot\dots\ot w_{2d})$
is a linear combination of the functions
$\delta_\pi$, restricted to $\id_V^{\ot d}\ot T^{2d}(V)$. But on this set, it is clear from
the formula above that $\delta_\pi=\pm \delta\circ\pi$, and the Theorem
is proved.
\end{proof}

Theorem \ref{thm:fft-osp} can be sharpened slightly.
In the case $r=2$, we let
\begin{eqnarray}\label{eq:singlet}
\hat{c} = \sum_{a, b=1}^{m+2n}\phi_{e_a}\otimes \eta^{a b} \phi_{e_b},
\end{eqnarray}
where $\eta^{a b}$ is the entry of $\eta^{-1}$ in the $a$-th row and $b$-th column.
Then $\hat{c}(v\otimes w)=(v, w)$ for all
$v, w\in V$. Note that $\zeta^*(\hat{c})
=\sum_{a, b=1}^{m+2n}e_a\otimes \eta^{a b} \phi_{e_b}$ is the identity element of
$\bE$. Thus $\hat{c}\in (V^*\otimes V^*)^G$.
Since $\tau(\hat{c})=\hat{c}$,
the statement of the theorem reduces to $(V^*\otimes V^*)^G=\Lambda \hat{c}$.

For general $r=2d$, we note that $\delta=\hat{c}^{\otimes d}$, where $\delta$ is the function
\eqref{eq:contraction}. Thus $\left({V^*}^{\otimes 2d}\right)^G$ is spanned by
$\hat{c}^{\otimes d}\circ\pi=\rho^*_r(\pi).\hat{c}^{\otimes d}$ with $\pi\in\Sym_{2d}$.

Let $\iota$ be the involution $(1,2)(3,4)\dots(2d-1,2d)\in\Sym_{2d}$. Then the centraliser
$Z_{\Sym_{2d}}(\iota)\cong \Sym_d\ltimes\Z_2^d$ is the precise stabiliser in $\Sym_{2d}$
of $\hat{c}^{\otimes d}$.


\begin{remark}\label{rem:diagbij}
Let $\cT_r:=\Sym_{2d}/Z_{\Sym_{2d}}(\iota)$.
Then $\cT_r$ can be identified with
the set of all pairings of the elements of $\{1,2,\dots,2r\}$,
i.e. the partitioning of $\{1,2,\dots,2r\}$ into
a disjoint union of pairs.
\end{remark}
Now the nontrivial case of Theorem \ref{thm:fft-osp} can be re-stated as follows.
\begin{corollary}\label{cor:fft-osp}
The subspace $\left({V^*}^{\otimes 2d}\right)^{\OSp(V)}$ of invariants is spanned
over $\Lambda$ by the set $\{\gamma_{[\sigma]}=\hat{c}^{\otimes d}\circ\sigma\mid [\sigma]\in\cT_r\}$, where
$\sigma$ is a representative of $[\sigma]$.
\end{corollary}

\subsection{FFT in polynomial form}\label{sect:osp-poly}

Let $\C^{p|q}$ be the $\Z_2$-graded vector space with $\sdim \C^{p|q}=(p,q)$, and
let $\Lambda^{p|q}:=\C^{p|q}\otimes_\C\Lambda$. We continue to write $V$ for the
orthosymplectic space with $\sdim V_\C=(m|2n)$.

Define $\cH:=\Hom_\Lambda(\lpq,V)$. This is a $\Z_2$-graded free $\La$-module which may be identified with
a space of matrices as in \S\ref{sect:lin-alg}.

We briefly recall this identification. If $e_1,\dots,e_{p+q}$ is
the standard homogeneous basis of $\C^{p|q}$, it is also a homogeneous $\Lambda$-basis of $\lpq$.
We similarly take $b_1,\dots,b_{m+2n}$ to be a homogeneous $\La$-basis of $V$. If  $w=\sum_{i=1}^{p+q}e_i\la_i$,
we say that the coordinate vector
$M(w)=
\begin{pmatrix} \la_1\\ \la_2\\ \cdot\\ \cdot\\ \la_{p+q}
\end{pmatrix}
$.
If $\phi\in\cH$, define the matrix $M(\phi)=(a_{ij})\in\cM(m|2n\times p|q;\La)$ by
$\phi(e_i)=\sum_{j=1}^{m+2n}b_ja_{ji}$. It is then easily verified that (cf. Lemma
\ref{lem:superlin}) in the obvious notation, $M(\phi(w))=M(\phi)M(w)$, and that
the map $\phi\mapsto M(\phi)$ is an isomorphism: $\cH\overset{\sim}{\lr}\cM(m|2n\times p|q;\La)$.
Note that
$$
A=(a_{ij})=M(\phi)=(u_1,\dots,u_{p+q}),
$$
 where the $u_i$ are the coordinate vectors of the $\phi(e_i)\in V$.
This provides an explicit identification of $\cH$ with $V\oplus V\oplus\dots\oplus V$ ($p+q$ summands).

Moreover the even elements of this space (i.e. $\cH_{\bar0}$) may be identified as those matrices
$A\in\cM(m|2n\times p|q;\La)$ with the entries of $A_{ij}$ in $\La_{\bar0}$ if $i=j$,
and in $\La_{\bar1}$ if $i\neq j$, where, in the obvious notation,
$$
A=\bordermatrix{&p&q&\cr
                m&A_{11} & A_{12}\cr
                2n& A_{21} &  A_{22}\cr}.
$$

Denote by $\cP:=\cP[\cH_{\bar0}]$ the $\Lambda$-superalgebra (i.e. graded-commutative, graded $\La$-algebra)
of polynomial functions on $\cH_{\bar0}$.
Given a function $f\in\cP[\cH_{\bar0}]$, we write $f(A)$ as
$
f(A)=f(u_1, u_2, \dots,  u_{p+q})
$
where $A=(u_1\, u_2\,\dots\,  u_{p+q})$ as explained above.

Note that $\GL(V)$ acts on $\cH$, preserving $\cH_{\bar0}$, via $(g\phi)(w)=g(\phi(w))$,
for $g\in\GL(V)$, $\phi\in\cH$ and $w\in \lpq$. This action transfers to a degree-preserving
action on $\cP$: if $f\in\cP$,
then $gf(\phi):=f(g\inv \phi)$. If $G$ is a subgroup of $\GL(V)$, then of course the above action
may be restricted to the subgroup $G$. We shall apply this with $G=\OSp(V)$.
It is evident that the subset $\cP[\cH_{\bar0}]^G$ of polynomial functions invariant
under $G=\OSp(V)$ forms a sub-$\La$-superalgebra of $\cP$.

Observe that we have the elements $f_{ij}$ ($1\leq i\leq j\leq p+q$) in the subalgebra
$\cP^2[\cH_{\bar0}]^G$, where in the above notation,
$$
f_{ij}(u_1,\dots,u_{p+q})=(u_i,u_j).
$$
The invariance of these functions is evident since $\OSp(V)$ is the isometry group of the
form $(-,-)$. The fact that $f_{ij}\in\cP^2$ follows from an argument similar to the proof
of Lemma \ref{lem:polyscal}.

The following statement may easily be shown to be equivalent to Theorem \ref{thm:fft-osp}.
 \begin{corollary}\label{cor:fft-osp-poly}
 Keep notation as above. Let $G=\OSp(V)$. Then $\cP[\cH_{\bar0}]^G$ is generated by
the functions $f_{ij}$ above.

More explicitly,
If $f\in \cP[\cH_{\bar0}]^G$, then for $A=(u_1\, u_2\,\dots\,  u_{p+q})\in \cH_{\bar0}$,
 $f$ is a polynomial in the quadratics $(u_a, u_b)$ $(a, b=1, 2, \dots, p+q)$.
 \end{corollary}
We should point out that Corollary \ref{cor:fft-osp-poly} is equivalent to
Segeev's \cite[Theorem 5.3]{S1}.

\section{Brauer diagrams and endomorphism algebras of $\OSp(V)$}

In this section, we describe the relationship between the Brauer algebra
of degree $r$ and endomorphism algebra $\End_{\OSp(V)}(V^{\otimes_\Lambda r})$.
We shall do this by working within the framework of the Brauer category
introduced in \cite{LZ5}.

\subsection{Categorical version of the first fundamental theorem}

The Brauer category \cite{LZ5} can be generalised to the setting of the
Grassmann algebra $\Lambda$.
Fix once and for all an element $\delta$ of $\Lambda_{\bar 0}$.
Denote by ${B}_k^\ell(\delta)$ the free
$\Lambda$-module with an even homogeneous basis consisting of
$(k, \ell)$ Brauer diagrams. Then the composition and tensor product of
Brauer diagrams can be extended to $\Lambda$-bilinear maps
\begin{eqnarray}\label{eq:products}
\begin{aligned}
&&\text{composition} & \quad \circ:  & & B_\ell^p(\delta)
\times B_k^\ell(\delta)\longrightarrow B_k^p(\delta),  \\
&&\text{tensor product} & \quad  \otimes: & & B_p^q(\delta)
\times B_k^\ell(\delta)\longrightarrow B_{k+p}^{q+\ell}(\delta).
\end{aligned}
\end{eqnarray}
The {\em Brauer category} $\cB(\delta)$
is the pre-additive small category equipped with the tensor product bi-functor $\otimes$
such that
\begin{enumerate}
\item the set of objects is $\N=\{0, 1, 2, \dots\}$,  and for any pair of objects $k, l$,
$\Hom_{\cB(\delta)}(k, l)$ is the $\Lambda$-module $B_k^l(\delta)$; the composition of morphisms is
given by the composition of Brauer diagrams;
\item the tensor product $k\otimes l$ of objects  $k, l$ is  $k+l$ in $\N$, and the
tensor product of morphisms is given by the tensor product of Brauer diagrams.
\end{enumerate}

Theorem 2.6 in  \cite{LZ5}  is still valid in the present setting.
The four elementary Brauer diagrams
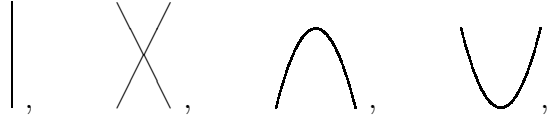
\begin{center}
\begin{picture}(205, 40)(-5,0)
\put(0, 0){\line(0, 1){40}}
\put(5, 0){,}

\put(40, 0){\line(1, 2){20}}
\put(60, 0){\line(-1, 2){20}}
\put(65, 0){,}

\qbezier(100, 0)(115, 60)(130, 0)
\put(135, 0){,}

\qbezier(170, 30)(185, -30)(200, 30)
\put(200, 0){,}
\end{picture}
\end{center}
which will be denoted by $I$, $X$, $A$ and $U$ respectively,
generate all Brauer diagrams by composition and tensor product.
The complete set of relations among these generators is that described in
\cite[Theorem 2.6(2)]{LZ5}.

\begin{lemma}{\rm{(\cite[Corollary 2.16]{LZ5})}}\label{lem:B-iso}
Let \[
\begin{aligned}
&A_q=A\circ (I\ot A\ot I)\dots (I^{\ot (q-1)}\ot A\ot I^{\ot (q-1)}),\\
&U_q=(I^{\ot (q-1)}\ot U\ot I^{\ot (q-1)})\circ\dots\circ (I\ot U\ot I)\circ U,\\
&I_q=I^{\ot q}.
\end{aligned}
\]
For all $p, q$ and $r$, define the $\Lambda$-linear maps
\[
\begin{aligned}
&{\mathbb U}_p^q=(-\otimes I_q)\circ(I_p\otimes U_q):
B_{p+q}^r(\delta)\longrightarrow B_p^{r+q}(\delta)\\
&{\mathbb A}^r_q=(I_{r+q}\otimes A_q)\circ(- \otimes I_q):
B_p^{r+q}(\delta)\longrightarrow B_{p+q}^r(\delta).
\end{aligned}
\]
Then ${\mathbb U}_p^q$ and ${\mathbb A}^r_q$ are mutually inverse.
\end{lemma}

Let $V=V_\C\otimes \Lambda$ be the orthosymplectic superspace defined in
Section \ref{sect:orthosym} with $\sdim V_\C=(m|2n)$.
We consider the following category of $\OSp(V)$-representations.
\begin{definition}
Let $G=\OSp(V)$. We denote by $\cT_G(V)$ the full subcategory of $G$-modules
with objects $V^{\otimes_\Lambda r}$ ($r=0, 1, \dots$).
The tensor product over $\Lambda$ of $G$-modules and
of $G$-equivariant maps is a bi-functor $\cT_G(V)\times \cT_G(V)\longrightarrow \cT_G(V)$.
Call $\cT_G(V)$ the {\em category of tensor representations of $G$}.
\end{definition}

Now we consider some morphisms in $\cT_G(V)$. Let $c_0=(\phi^{-1}\otimes\phi^{-1})(\hat{c})$,
where $\hat{c}\in V^*\otimes V^*$ is defined by \eqref{eq:singlet}.
Then $c_0=\sum_{a,b=1}^{m+2n} e_a\otimes\eta^{a b} e_b$, which
is canonical in that it is independent of the basis $\cE=(e_1, e_2, \dots, e_{m+2n})$,
and is invariant under $\OSp(V)$. Consider the $\Lambda$-linear maps
\begin{eqnarray}\label{P-C-C}
\begin{aligned}
&\tau: V\otimes V\longrightarrow V\otimes V, \quad v\otimes w \mapsto (-1)^{[v][w]}w\otimes v, \\
&\check{C}: \Lambda \longrightarrow V\otimes V,  \quad 1\mapsto c_0, \\
&\hat{C}: V\otimes V \longrightarrow \Lambda, \quad v\otimes w\mapsto (v, w).
\end{aligned}
\end{eqnarray}

The following result is the analogue of \cite[Lemma 3.1]{LZ5} in the present setting.
\begin{lemma} \label{lem:PAU}
Let $G=\OSp(V)$ and $d=m-2n$, where $\sdim V_\C=(m|2n)$, and denote
 by $\id$ the identity map on $V$.
Then the maps $\tau$, $\check{C}$ and $\hat{C}$ are all $G$-equivariant. Furthermore,
they satisfy the following relations:
\begin{eqnarray}
&&\tau^2=\id^{\ot 2}, \quad
(\tau\ot \id)(\id\ot \tau)(\tau\ot \id) = (\id\ot \tau)(\tau\ot\id)(\id\ot \tau),
\label{eq:PPP}\\
&&\tau \check{C} =  \check{C}, \quad
\hat{C} \tau =   \hat{C}, \label{eq:fse-es}\\
&&\hat{C}\check{C}=d,  \quad (\hat{C}\ot\id)(\id\ot\check{C})=\id=(\id\ot\hat{C})(\check{C}\ot\id), \label{eq:CC}\\
&&(\hat{C}\ot\id)\circ (\id\ot \tau)= (\id\ot\hat{C})\circ (\tau\ot\id), \label{eq:CPC-I}\\
&&(\tau\ot \id)\circ(\id\ot \check{C})=(\id\ot \tau)\circ(\check{C}\ot \id). \label{eq:CPC-P}
\end{eqnarray}
\end{lemma}
\begin{proof}The first statement is clear.
Now equation \eqref{eq:PPP} reflects standard properties of permutations,
and the relations \eqref{eq:fse-es} are evident.
We have
\[
\begin{aligned}
\hat{C}\check{C}&=\hat{C}(\sum_{a, b} e_a \ot\eta^{a b} e_b)
                =\sum_{a, b} \eta^{a b} (e_a,  e_b)
                =\sum_{a, b} \eta^{a b} \eta_{a b}
                =d.
\end{aligned}
\] This proves the first relation of \eqref{eq:CC}.
The proofs of the remaining relations are similar, and therefore omitted.
\end{proof}

We have the following result.
\begin{theorem}\label{thm:functor} Let $G=\OSp(V)$ and $d=m-2n$, where $\sdim V_\C=(m|2n)$.
There is a unique additive covariant functor $F: \cB(d) \longrightarrow \cT_G(V)$
of pre-additive categories with the following properties:
\begin{enumerate}
\item[(i)] $F$ sends the object
$r$ to $V^{\otimes_\Lambda r}$ and morphism $D: k \to \ell$ to
$F(D): V^{\otimes_\Lambda k}\longrightarrow V^{\otimes_\Lambda l}$ where $F(D)$
is defined on the generators of Brauer diagrams by
\begin{eqnarray}\label{eq:F-generating}
\begin{aligned}
F\left(
\begin{picture}(30, 20)(0,0)
\put(15, -15){\line(0, 1){35}}
\end{picture}\right)=\id_V,
\quad&
F\left(
\begin{picture}(30, 20)(0,0)
\qbezier(5, -15)(15, 3)(25, 20)
\qbezier(5, 20)(15, 3)(25, -15)
\end{picture}\right) = \tau, \\
F\left(
\begin{picture}(30, 20)(0,0)
\qbezier(5, 20)(15, -50)(25, 20)
\end{picture}\right) = \check{C}, \quad&
F\left(
\begin{picture}(30, 20)(0,0)
\qbezier(5, -15)(15, 50)(25, -15)
\end{picture}\right) = \hat{C};
\end{aligned}
\end{eqnarray}
\item[(ii)] $F$ respects tensor products, thus for any
objects $r, r'$ and morphisms $D, D'$ in $\cB(d)$,
\[
\begin{aligned}
&F(r\otimes r')=V^{\otimes_\Lambda r}\otimes_\Lambda V^{\otimes_\Lambda r'}=F(r)\otimes_\Lambda F(r'), \\
&F(D\otimes D')= F(D)\otimes_\Lambda F(D').
\end{aligned}
\]
\end{enumerate}
\end{theorem}
\begin{proof}
The proof is the same as in \cite{LZ5}. Here we merely point out that
the essential part of the proof is to show that the maps in \eqref{eq:F-generating}
satisfy all the defining relations of the elementary Brauer diagrams given in
\cite[Theorem 2.6(2)]{LZ5}. This follows from Lemma \ref{lem:PAU}.
\end{proof}

The following result is \cite[Lemma 3.6]{LZ5} adapted to the present context. We omit the
details of the proof.
\begin{lemma}\label{lem:transfrom}
Let $H_s^t = \Hom_G(V^{\otimes s}, V^{\otimes t})$ for all $s, t\in \N$.
\begin{enumerate}
\item The $\Lambda$-linear maps
\[
\begin{aligned}
&F{\mathbb U}_p^q:=(-\otimes \id_V^{\otimes q})(\id_V^{\otimes p}\otimes F(U_q)):
H_{p+q}^r \longrightarrow H_p^{r+q}, \\
&F{\mathbb A}^r_q:=(\id_V^{\otimes r}\otimes F(A_q))(- \otimes \id_V^{\otimes q}):
H_p^{r+q} \longrightarrow H_{p+q}^r
\end{aligned}
\]
are well defined and are mutually inverse isomorphisms.
\item For each pair $k, \ell$ of objects in $\cB(d)$, the functor $F$ induces a $\Lambda$-linear map
\begin{eqnarray}\label{eq:functionF}
\begin{aligned}
{F}_k^\ell: B_k^\ell(d)\longrightarrow H_k^\ell=\Hom_G(V^{\otimes k}, V^{\otimes \ell}), \quad
D \mapsto F(D),
\end{aligned}
\end{eqnarray}
and the following diagrams are commutative.
\begin{displaymath}
    \xymatrix{
        B_p^{r+q}(d) \ar[r]^{{\mathbb A}^r_q} \ar[d]_{F_p^{r+q}} &  B_{p+q}^r(d) \ar[d]^{F_{p+q}^r} \\
         H_p^{r+q}  \ar[r]_{F{\mathbb A}^r_q}  & H_{p+q}^r}
\quad\quad
\xymatrix{
         B_{p+q}^r(d) \ar[r]^{{\mathbb U}_p^q} \ar[d]_{F_{p+q}^r} &  B_p^{r+q}(d)  \ar[d]^{F_p^{r+q}} \\
         H_{p+q}^r \ar[r]_{F{\mathbb U}_p^q}  & H_p^{r+q}. }
\end{displaymath}
\end{enumerate}
\end{lemma}

The first fundamental theorem of invariant theory for $\OSp(V)$
can now be re-interpreted as follows
\begin{theorem} \label{thm:fft-cat}  The functor $F: \cB(d)\longrightarrow \cT_{\OSp(V)}(V)$ is full.
That is, $F$ is surjective on $\Hom$ spaces.
\end{theorem}
\begin{proof}
Lemma \ref{lem:transfrom} gives a canonical
isomorphism $B_k^\ell(d)\simeq B_{k+\ell}^0(d)$, thus we only need to study
$F_{k+\ell}^0$. When $k+\ell$ is odd, the theorem is trivially true.
Thus we only need to consider the case $k+l=2r$.

By Corollary \ref{cor:fft-osp}, every element of $H_{2r}^0$ is a $\Lambda$-linear
combination of functionals $\gamma_\alpha$ for $[\alpha]\in\cT_r$. As remarked in
Remark \ref{rem:diagbij}, the elements of $\cT_r$ are in canonical bijection
with pairings of the set $\{1,2,\dots,2r\}$. Thus there is a
one to one correspondence between the elements of $\cT_r$ and $(2r, 0)$
Brauer diagrams. Let $D$ be the diagram corresponding to $[\alpha]\in\cT_r$.
Then $F(D)=\gamma_\alpha$.
Thus $F_{2r}^0$ is surjective, proving the theorem.
\end{proof}

\subsection{The Brauer algebra and endomorphism algebras}

For any object $r$ in $\cB(\delta)$,  the set of morphisms
$B_r^r(\delta)$ from $r$ to itself form
a unital associative superalgebra under composition
of Brauer diagrams. This is the Brauer superalgebra of degree
$r$ with parameter $\delta$, which will be denoted by $B_r(\delta)$.
If $\delta\in\C$, we let $B_r(\delta)_\C$ be the Brauer algebra
over $\C$, then $B_r(\delta)=B_r(\delta)_\C\otimes_\C\Lambda$,
where $B_r(\delta)_\C$ is regarded as purely even.

For $i=1,\dots,r-1$, let $s_i$ and $e_i$ respectively be the $(r, r)$ Brauer diagrams shown in
Figure \ref{s-e}.
\begin{figure}[h]
\begin{center}
\begin{picture}(350, 60)(0,0)
\put(0, 0){\line(0, 1){60}}
\put(40, 0){\line(0, 1){60}}
\put(18, 30){...}
\put(10, 0){$i-1$}

\qbezier(60, 0)(70, 30)(80, 60)
\qbezier(60, 60)(70, 30)(80, 0)

\put(100, 0){\line(0, 1){60}}
\put(140, 0){\line(0, 1){60}}
\put(118, 30){...}
\put(150, 0){, }

\put(200, 0){\line(0, 1){60}}
\put(240, 0){\line(0, 1){60}}
\put(218, 30){...}
\put(210, 0){$i-1$}

\qbezier(260, 0)(270, 45)(280, 0)
\qbezier(260, 60)(270, 15)(280, 60)

\put(300, 0){\line(0, 1){60}}
\put(340, 0){\line(0, 1){60}}
\put(318, 30){...}
\put(350, 0){ \  }
\end{picture}
\end{center}
\caption{ }
\label{s-e}
\end{figure}
Then $B_r(\delta)$ as superalgebra over $\Lambda$ is generated by $\{s_i, e_i \mid i=1, 2, \dots, r-1\}$
with the standard relations given in \cite[Lemma 2.18(1)]{LZ5}. We note in particular that
the elements $s_i$ generate the subalgebra $\Lambda\Sym_r\subset B_r(\delta)$.

The following result is an immediate corollary of Theorem \ref{thm:fft-cat}.
\begin{corollary} \label{cor:Brauer-super}
Let $d=m-2n$. Then
\[\End_{\OSp(V)}(V^{\otimes r}) = F_r^r(B_r(d))\]
for all $r$ as associative superalgebra.
\end{corollary}

Recall from Remark \ref{rem:HC-pair} that the orthosymplectic supergroup $\OSp(V)$ may be defined
by a Harish-Chandra super pair \cite{DM}.
The following result follows from the corollary.
\begin{corollary}\label{cor:Brauer}
Let $G:=(G_0, \osp(V_\C))$, with $G_0={\rm O}((V_\C)_{\bar0})\times{\rm Sp}((V_\C)_{\bar1})$,
be  the Harish-Chandra super pair defining
the orthosymplectic supergroup $\OSp(V)$.
Denote by $\End_G(V_\C^{\otimes r})$ the subspace of $\End_\C(V_\C^{\otimes r})$
consisting of $G$-invariants, that is, elements which are
both $G_0$-invariant and $\osp(V_\C)$-invariant. Then
\[
\End_G(V_\C^{\otimes r}) = F_r^r(B_r(d)_\C),
\]
where $B_r(d)_\C$ is the complex Brauer algebra of degree $r$ with parameter $d=m-2n$.
\end{corollary}
\begin{proof}
Let $E_i=\End_\C(V_\C)\otimes\Lambda_i$, were $\Lambda_i$ is the homogeneous component of $\Lambda$ of
degree $i$ relative to the $\Z_+$-grading. Set $E_+=\sum_{i\ge 1} E_i$. Then all elements of
$E_+$ are nilpotent endomorphisms on $V$. Let $\osp(V)_+=\osp(V)\cap E_+$. Recall the specialisation map
$\cR: \Lambda\longrightarrow\C$ defined by \eqref{eq:specialisation}. We have the following
surjective group homomorphism
$\hat\cR=\id_{\End_\C(V_\C)}\otimes\cR: G\longrightarrow G_0$.
Given $g\in\OSp(V)$, we denote $g_0=\hat\cR(g)$, then $\hat\cR(g_0^{-1}g)=1$. Thus there
exists some $X\in\osp(V)_+$ such that $g_0^{-1}g=\exp(X)$, that is, $g=g_0\exp(X)$.
In fact, $X=-\sum_{i\ge 1} (1-g_0^{-1}g)^i$,
which is a finite sum since $1-g_0^{-1}g$ is a nilpotent endomorphism on $V$.

An element $\phi\in \End_\Lambda(V^{\otimes r})$ is $G$-invariant, that is,
$g.\phi =g \phi g^{-1} =\phi$ for all $g=g_0\exp(X)\in G$, if and only if both of the following conditions are satisfied:
\[
\begin{aligned}
&(i) &  &g_0.\phi =\phi, \quad \text{for all $g_0\in G_0$,  and }\\
&(ii) & &\exp(X).\phi=\phi, \quad \text{for all $X\in \osp(V)_+$},
\end{aligned}
\]
where the second condition is equivalent to $(ii)'$ $X.\phi=0$ for all $X\in \osp(V)_+$.
Write $\phi=\sum_i \phi_i\otimes\lambda_i$,
where $\phi_i\in \End_\C(V_\C)$ and $\lambda_i\in\Lambda_i$ for all $i$.
Then condition $(i)$ is clearly equivalent to $\phi_i\in \End_{G_0}(V_\C^{\otimes r})$ for all $i$.
Since $\Lambda$ is of infinite degree, condition $(ii')$ is satisfied if and only if
$\phi_i\in \End_{\osp(V_\C)}(V_\C^{\otimes r})$ for all $i$. This completes the proof.
\end{proof}
\begin{remark}\label{rem:cel}
It is well known that the Brauer algebras are cellular \cite{GL96}, and in fact that
the Brauer category is a cellular category. This makes it possible to apply cellular theory to
the decomposition of the tensor powers $V^{\ot r}$
(cf. \cite{GL98,GL03,GL04}, where the possibility of different
ground rings is discussed).
\end{remark}
\section{Proof of the key lemma}\label{sect:pf-new}

In this section we shall prove the key Lemma \ref{lem:key}. We maintain the notation of Section \ref{sect:inv-osp},
in particular that of \S\ref{sect:osp-poly}.

\subsection{The key lemma in terms of super matrices} We begin by translating the statement into one
concerning the space
of super matrices
described in Lemma \ref{lem:superlin}. Recall (\S\ref{sect:orthosym}) that with respect to the standard
 basis of $V$, the matrix
of the orthosymplectic form $(-,-)$ on $V$ is
$
\eta=\bordermatrix{&m&2n&\cr
                m&I & 0\cr
                2n& 0 &  J\cr},
$
so that the transpose
$
\eta^t=\eta\inv=\bordermatrix{&m&2n&\cr
                m&I & 0\cr
                2n& 0 & - J\cr}
$.
Recall also (\S\ref{sect:osp-poly}) that if $e_1,\dots,e_{m+2n}$ is the standard ($\La$-)basis of $V$ and
$v=e_1\la_1+\dots+e_{m+2n}\la_{m+2n}\in V$, the coordinate vector
$M(v) =
\begin{pmatrix} \la_1\\ \la_2\\ \cdot\\ \cdot\\ \la_{m+2n}
\end{pmatrix}
$.
Similarly, for $\alpha\in\End_\La(V)$, we have the matrix $M(\alpha)\in\cM(m|2n;\La)$.

Define the {\it supertranspose} of this column vector by
$$
\begin{pmatrix} \la_1\\ \la_2\\ \cdot\\ \cdot\\ \la_{m+2n}
\end{pmatrix}\st=
((-1)^{[1][\lambda_1]} \la_1,(-1)^{[2][\lambda_2]} \la_2, \cdots, (-1)^{[m+2n][\lambda_{m+2n}]}\la_{m+2n}),
$$
where $[i]:=[e_i]$ ($={\bar0}$ if $1\leq i\leq m$, and $\bar1$ otherwise).
It is then evident that for $v,w\in V$,
\be\label{eq:matrixip}
(v,w)=M(v)\st\eta M(w).
\ee

Similarly, if $A=(a_{ij})\in\cM(m|2n;\La)$, define the supertranspose $A\st$ of $A=(a_{ij})$ by
\be\label{eq:st}
(A\st)_{ij}=(-1)^{[j][a_{ji}]}a_{ji},
\ee
and recall that under the isomorphism $\End_\La(V)\overset{\sim}{\lr}\cM(m|2n;\La)$,  $\bE=\End_\La(V)_{\bar0}$
is mapped to $\cM(m|2n;\La)_{\bar0}:=\{A=(a_{ij})\mid [a_{ij}]=[i]+[j]\}$. It follows that for
$A\in\cM(m|2n;\La)_{\bar0}$, $(A\st)_{ij}=(-1)^{[j]([i]+[j])}a_{ji}$.

The following relations are now easily verified.

For $A\in\cM(m|2n;\La)_{\bar0}$ and $B\in\cM(m|2n;\La)$, we have
\be\label{eq:compst}
(AB)\st=B\st A\st.
\ee
Let $\alpha\in \bE$, and let $A=M(\alpha)$. Then
\be\label{eq:dagmat}
M(\alpha^\dag):=A^\dag=\eta\inv A\st \eta.
\ee
This leads to the following characterisation of the space of self adjoint even supermatrices.
\begin{lemma}\label{lem:char-sa}
Let $A=\bordermatrix{&m&2n&\cr
                m&A_{11} & A_{12}\cr
                2n& A_{21} &  A_{22}\cr}\in\cM(m|2n;\La)_{\bar0}$. Then $A^\dag=A$ if and only if the matrix
$B=\eta A=\bordermatrix{&m&2n&\cr
                m&B_{11} & B_{12}\cr
                2n& B_{21} &  B_{22}\cr}
=\bordermatrix{&m&2n&\cr
                m&A_{11} & A_{12}\cr
                2n& JA_{21} &  JA_{22}\cr}
$ satisfies
\begin{enumerate}
\item $B_{11}^t=B_{11}$,

\item $B_{22}^t=-B_{22}$, and

\item $B_{21}=B_{12}^t$.
\end{enumerate}
\end{lemma}
\begin{proof}
Write $\cM=\cM(m|2n;\La)$, so that $\cM_{\bar0}\simeq \bE$; let $A=\bordermatrix{&m&2n&\cr
                m&A_{11} & A_{12}\cr
                2n& A_{21} &  A_{22}\cr}\in\cM_{\bar0}$.
Then $A\st=\bordermatrix{&m&2n&\cr
                m&A_{11}^t & -A_{21}^t\cr
                2n& A_{12}^t &  A_{22}^t\cr}$
and
$A^\dag=\bordermatrix{&m&2n&\cr
                m&A_{11}^t & -A_{21}^tJ\cr
                2n& JA_{12}^t &  -JA_{22}^tJ\cr}$.
The condition $A^\dag=A$ may be written $A\st\eta=\eta A$, and an easy computation shows that
this condition translates into the stated conditions on the four submatrices of $B=\eta A$.
\end{proof}

Evidently $\End_\La(V)^+=\{T\in\End_\La(V)\mid T^\dag=T\}\cong \{A\in\cM\mid A^\dag=A\}$.
 We therefore write $\bE^+=\{A\in\cM_{\bar0}\mid A^\dag=A\}$.
The conditions that $A^\dag=A\in\bE^+$ may be rephrased as follows.
\be\label{eq:sa}
A_{11}^t=A_{11};\;\;A_{22}^t=-JA_{22}J;\;\;\text{ and }A_{21}=-JA_{12}^t.
\ee
We then have, for $A\in \bE$,
\be\label{eq:omega}
\omega(A)=A^\dag A\in\bE^+.
\ee

If $A=\bordermatrix{&m&2n&\cr
                m&A_{11} & A_{12}\cr
                2n& A_{21} &  A_{22}\cr}\in\cM_{\bar0}\simeq\bE,$ then $S=\omega(A)
=\bordermatrix{&m&2n&\cr
                m&S_{11} & S_{12}\cr
                2n& S_{21} &  S_{22}\cr}$, where
\be\label{eq:ada}
\begin{aligned}
S_{11}=A_{11}^tA_{11}+A_{12}^tJ^tA_{22}; \;\;\;& S_{12}=A_{11}^tA_{12}+A_{21}^tJ^tA_{22};\\
S_{21}=J^tA_{12}^tA_{11}+J^tA_{22}^tJA_{21};\;\;\; &
S_{22}=J^tA_{12}^tA_{12}+J^tA_{22}^tJA_{22}.\\
\end{aligned}
\ee

Now let $\bS$ be the space of matrices $B$ satisfying the conditions of Lemma \ref{lem:char-sa}.
Evidently the map $s:A\mapsto \eta A$ defines an isomorphism $s:\bE^+\lr \bS$.

\begin{definition}\label{def:omega0} We define $\omega_s:\bE\lr \bS$ by $\omega_s=s\circ\omega$.

\end{definition}

It is clear that Lemma \ref{lem:key} is equivalent to the following statement.

\begin{lemma}
If $f\in\cP[\bE]$ is such that $f(gA)=f(A)$ for $g\in\OSp(V)$ and $A\in\bE$, then there exists
$F\in\cP[\bS]$ such that $f=\omega_s^*(F)=F\circ\omega_s$.
\end{lemma}

\subsection{Further reformulation of the Key Lemma-notation}  Let $\phi_{ij}\in \End_\La(V)^*$
be the $i,j$ coordinate function on $\bE$. Then $[\phi_{ij}]=[i]+[j]$, and the super-matrix $A\in\bE$ if and only if
$[a_{ij}]=[i]+[j]$ for all $i,j$.

Let $c_j$ be the $j^{\text{th}}$ column of $A\in\bE$, and let $v_j\in V=V_\C\ot_\C\La$ be the element
with coordinate vector
$c_j$. It is then clear that
\[
\begin{cases}
v_j\in V_{\bar0}\text{ if }1\leq j \leq m\\
v_j\in V_{\bar1}\text{ if }m+1\leq j\leq m+2n.\\
\end{cases}
\]
We therefore may, and shall, denote elements $A\in\bE$ by $A=[v_1, v_2,\cdots, v_{m+2n}]$ with the above
convention. It is then easily seen that
\be\label{eq:omegas}
\omega_s(A)=\big( (v_i,v_j)\big).
\ee

With this notation, the key lemma may now be stated as follows.
\begin{lemma}\label{lem:key3}
Let $f\in\cP[\bE]$ be such that $f([gv_1,\cdots,gv_{m+2n}])=f([v_1,\cdots,v_{m+2n}])$ for
all $g\in\OSp(V)$ and $[v_1,\cdots,v_{m+2n}]\in\bE$.
Then $f$ is a polynomial in the variables $(v_i,v_j)$.
\end{lemma}

We shall denote the coordinate functions on $\bS$ by $u_{ij}$. Then $[u_{ij}]=[i]+[j]$,
and if $A=[v_1,\cdots,v_{m+2n}]\in\bE$, then \eqref{eq:omegas} says that $u_{ij}(\omega_s(A))=(v_i,v_j)$. Note that
by Lemma \ref{lem:char-sa},
if either $i\leq m$ or $j\leq m$, then $u_{ij}=u_{ji}$, while if $m+1\leq i,j\leq m+2n$, then
$u_{ij}=-u_{ji}$.

\subsection{Preliminary results for the proof of Lemma \ref{lem:key3}} The proof, which will be along the lines
of the proof of \cite[Proposition 5.2.6]{GW} (which itself is an adaptation of the approach of \cite[Appendix]{ABP}),
is by induction on $\sdim(V_\C)$.
We begin with some preliminary results.

\begin{lemma}\label{lem:omegadense}
The image of $\omega:\bE\lr\bE^+$ is dense in $\bE^+$.
\end{lemma}
We shall give two proofs, the second being the more concrete.
\begin{proof}[First proof.]
We show that $\omega(\bE)$ and $\bE^+$ are superschemes of the same dimension. To see this, note first that
the generic fibre of $\omega$ is $G=\OSp(V)$, since over any invertible element $S\in\bE^+$, $\omega\inv(S)
\cong G$. It follows that $\dim \omega(\bE)=\dim(\bE)-\dim(G)$. Further, we have seen that
$\End(V)=\End(V)^+\oplus\End(V)^-$, where $\End(V)^{\pm}=\{T\in\End(V)\mid T^\dag=\pm T\}$.

But the (Cayley) map $C:g\mapsto (1-g)\inv(1+g)$ maps $G^0=\OSp(V)^0$ birationally into $\bE^-$, where
$G^0$ is the connected component of the identity in $G$. It follows that
$\dim(G)=\dim(G^0)=\dim(\bE^-)=\dim(\bE)-\dim(\bE^+)$, whence
$\dim(\bE^+ )=\dim(\bE)-\dim(G)=\dim(\omega(\bE))$.
\end{proof}

\begin{proof}[Second proof.]
Identify  $\bE$ with $\cM_{\bar0}$. Then
$\bE^+=\{A\in \cM_{\bar0}\mid A^\dag = A\}$ and $\bS$ is the set of $B\in\cM_{\bar0}$ which satisfy the
conditions of Lemma \ref{lem:char-sa}. Recall that $\bS=\eta\bE^+$.
Denote by $\bE^+_{\text{reg}}$ the subset of $\bE^+$  consisting of its nondegenerate elements, that is,
$
\bE^+_{\text{reg}}=\{ S\in\bE^+\mid \text{$\det\cR(S)\ne 0$} \}.
$
The argument of Proposition \ref{prop:poly-polyfn} shows that $\bE^+_{\text{reg}}$ is dense in $\bE^+$.

For any $S\in\bE^+_{\text{reg}}$, we have $\cR(S)=\begin{pmatrix}S_{00} & 0\\ 0 & S_{11}\end{pmatrix}$,
where $S_{00}$ and $S_{11}$ are invertible complex matrices of size $m$ and $2n$ respectively.
A well known Gram-Schmidt orthogonalisation argument over $\C$ shows
that there exists an (evidently invertible) complex matrix
$A_0=\begin{pmatrix}A_{00} & 0\\ 0 & A_{11}\end{pmatrix}$
such that
$
\cR(S) = A_0^\dag A_0.
$

Write $S_1= (A_0^\dag)^{-1} S A_0^{-1}$. Then $\cR(S_1)$ is
the identity matrix $I$, and the entries of
$
\Psi=S_1-I
$
all belong to the augmentation ideal  $\Lambda_{\bar 1}\Lambda$ of $\Lambda$. Hence $\Psi$ is nilpotent,
and the binomial expansion
\[
 (I + \Psi)^{1/2}:=I +\sum_{k=1}^\infty\binom{ \frac{1}{2}}{k} \Psi^k
\]
is finite.
Since $S_1\in \bE^+$, $\Psi\in\bE^+$. Thus
$A_1:=(I + \Psi)^{1/2}\in\bE^+$ and $S_1=A_1^2=A_1^\dag A_1$.  Now
$
S=A_0^\dag S_1 A_0 = A_0^\dag A_1^\dag A_1 A_0 = \omega(A_1A_0).
$
Therefore,  $\im(\omega)\supset\bE^+_{\text{reg}}$. This proves the density of $\im(\omega)$ in $\bE^+$.
\end{proof}


Next we prove another density result.
\begin{lemma}\label{lem:uwdense} Let $\La^*$ be the group of invertible elements of $\La$.
Then the set
$$
\begin{aligned}
U:=\{[v_1,\cdots,v_{m+2n}]\in\bE\mid (v_i,v_i)\in&\La^*\text{ for }1\leq i\leq m\text{ and }\\
(v_{m+2j-1},v_{m+2j})\in&\La^*\text{ for }1\leq j\leq n\}\\
\end{aligned}
$$
is dense in $\bE$.

\end{lemma}
\begin{proof}
We need to show that if $f\in\cP[\bE]$, then $f|_U=0$ implies that $f=0$. Note that $U$ is stable
under multiplication by $\C^*$. It follows that if $f$ vanishes on $U$, then each homogeneous component
of $f$ vanishes on $U$, and therefore we may assume that $f$ is homogeneous. Let $\phi_{ij}$ be the
coordinate functions on $\bE$. Then as in \eqref{eq:frv}, $f$ may be written

\be\label{eq:f}
f=\sum_{(m_1\dots,m_p)}\la_{m_1,\dots,m_{(m+2n)^2}}\phi_{11}^{m_1}\dots\phi_{i_kj_k}^{m_{k}}
\dots\phi_{m+2n,m+2n}^{m_{(m+2n)^2}},
\ee
where the sum is over all sequences $m_1,\dots,m_{(m+2n)^2}$, with $m_j\in\N$ for all $j$ and $0\leq m_k\leq 1$
for $k$ such that $[\phi_{i_kj_k}]=\bar1$, and where the indices
$(i_kj_k)$ are taken in some fixed order $k=1,2,\dots,(m+2n)^2$, and $\la_{m_1,\dots,m_{(m+2n)^2}}\in\La$.

Now the elements of $\bE$ may be regarded as matrices
\[
A:=[v_1,\dots,v_{m+2n}]=
\bordermatrix{&m&2n&\cr
                m&A_{11} & A_{12}\cr
                2n& A_{21} &  A_{22}\cr},
\]
where the entries of $A_{ij}$ lie in $\La_{[i]+[j]}$, the indices being taken modulo $2$.
Thus the entries of $A$ have parity described as shown below.
\[
A:=[v_1,\dots,v_{m+2n}]\in
\bordermatrix{&m&2n&\cr
                m&\La_{\bar0} & \La_{\bar1}\cr
                2n& \La_{\bar1} &  \La_{\bar0}\cr}.
\]

Evidently, the conditions that $A$ be in $U$ depend only on the reduction of $A$ modulo the augmentation ideal,
that is, on the specialisation $\cR(A)$ (cf. \eqref{eq:specialisation}), and since
$\cR(A_{12})=0$ and $\cR(A_{21})=0$, it follows there are no conditions imposed on the Grassmann variables.
Moreover the conditions that $A$ be in $U$ amount to an open condition on $\cR(A_{11})$ and $\cR(A_{22})$.

 We may now argue exactly as in the proof
of Proposition \ref{prop:poly-polyfn} to show that $f=0$.

\end{proof}

It is a trivial consequence of Lemma \ref{lem:uwdense} that any subset of $\bE$ which contains $U$ is also dense.

Recall that we have fixed a homogeneous $\C$ basis $e_1,\dots,e_{m+2n}$ of $V_\C$, such that
$[e_i]=[i]=\begin{cases}\bar0\text{ if }1\leq i\leq m\\ \bar1\text{ if }m+1\leq i\leq m+2n\\
\end{cases}$. This is of course also a homogeneous $\La$-basis of $V$.
\begin{remark}\label{rem:parity}
Notice that if $A:=[v_1,\dots,v_{m+2n}]\in\bE$, then from the form of the matrix above, it is evident
that 
\[
(v_i,v_j)\in\La_{\bar0}\text{ if }1\leq i,j\leq m\text{ or }m+1\leq i,j\leq m+2n.
\]
This will be used in the proofs below.
\end{remark}

\begin{lemma}\label{lem:baby}
Lemma \ref{lem:key3} is true when (i) $\sdim(V_\C)=(1|0)$ and when (ii) $\sdim(V_\C)=(0|2)$.
\end{lemma}
\begin{proof} As we have already observed, we may take $f$ homogeneous, and since $-1\in\OSp(V)$,
we may take $f$ to have even degree, say $2d$.

(i) Let $U=\{v\in V_{\bar0}\simeq\bE\mid (v,v)\in\La^*\}$. Then $U$ is dense in $V_{\bar0}$,
and for $v\in U$, we have $(v,v)=\la\in\La_{\bar0}^*$ (by Remark \ref{rem:parity}.
 Hence $\la=z+\xi$, where $z\in\C^\times$ and $\xi$ is nilpotent.
It follows that there is an element $\mu=\exp(\frac{-\log(\lambda)}{2})\in\La_{\bar0}^*$ such that $\mu^2=\la\inv$.
Thus $(\mu v,\mu v)=1$, so there is an element $g\in\OSp(V)$ such that $g(\mu v)=e_1$. Then
$f(v)=f(gv)=f(\mu\inv e_1)=\mu^{-2d}f(e_1)=\la^df(e_1)$. Hence for $v\in U$, $f(v)=(v,v)^df(e_1)$.
Since this is true on a dense subset of $\bE$ it holds for all $v\in\bE$.

(ii) Let $W=\{[v,w]\in\bE\mid (v,w)\in\La^*\}$. Then $W$ is dense in $\bE$ and for $[v,w]\in W$,
$(v,w)=\la\in\La_{\bar0}^*$. As in (i) we take $\mu\in\La^*$ such that $\mu^2=\la\inv$; then $(\mu v,\mu w)=1$,
and there is $g\in \OSp(V)$ such that $g(\mu v)=e_2$ and $g(\mu w)=e_1$ (cf. Lemmas \ref{lem:bs-constr}
and \ref{lem:bs-change}). Thus $f([v,w])=f([gv,gw])=f(\mu\inv e_2,\mu\inv e_1)=\mu^{-2d}f(e_2,e_1)=
\la^df(e_2,e_1)$. Thus for $[v,w]$ in the dense subset $W\subset\bE$, $f(v,w)=(v,w)^df(e_2,e_1)$.
By density this holds for all $[v,w]\in\bE$.
\end{proof}

\subsection{Completion of the proof of Lemma \ref{lem:key}} We shall prove Lemma \ref{lem:key3}
(which is equivalent to
Lemma \ref{lem:key}). As usual, we take $f\in\cP[\bE]$ homogeneous of degree $2d$ such that
for $g\in\OSp(V)$ and $A=[v_1,\cdots,v_{m+2n}]\in\bE$, $f([g(v_1),\cdots,g(v_{m+2n})])=f([v_1,\cdots,v_{m+2n}])$.

We shall deduce Lemma \ref{lem:key3} from the following result, which provides the inductive step for a proof by
induction on $\sdim(V)$.

\begin{lemma}\label{lem:ind} Let $f\in\cP^{2d}[\bE]$ be as above and
assume that the statement of Lemma \ref{lem:key3} holds
for all superspaces $W_\C$ with $\sdim(W)<(m|2n)$. Let $U\subset\bE$ be the set defined in Lemma \ref{lem:uwdense}.

(i) If $n>0$, then there is an integer $r$, a polynomial function $F\in\cP[\bS]$ and a subset $U_1\supseteq U$
such that for $A\in U_1$,
\be\label{eq:ng0}
f(A)=u_{m+1,m+2}^{-r}(\omega_s(A))F(\omega_s(A).
\ee

(ii) If $m>0$, then there is an integer $r_1$, a polynomial function $F_1\in\cP[\bS]$ and a subset $U_2\supseteq U$
such that for $A\in U_2$,
\be\label{eq:mg0}
f(A)=u_{11}^{-r_1}(\omega_s(A))F_1(\omega_s(A).
\ee
\end{lemma}

\begin{proof}[Proof of Lemma \ref{lem:ind}] We start with the proof of (i).

Assume that $n\geq 1$, and take $U_1:=\{[v_1,\cdots,v_{m+2n}]\in\bE\mid (v_{m+1},v_{m+2})\in\La^*\}$.
Take $A=[v_1,\cdots,v_{m+2n}]\in U_1$, and let $(v_{m+1},v_{m+2})=\la\in\La^*$. Then as above,
$\la\in\La_{\bar0}^*$.
If $\mu^2=-\la\inv$, then $(v_{m+1}\mu,v_{m+2}\mu)=-1$, whence by Lemmas \ref{lem:bs-constr} and
 \ref{lem:bs-change} there is
an element $g_0\in\OSp(V)$ such that $g_0(v_{m+1}\mu)=e_{m+1}$ and $g_0( v_{m+2}\mu)=e_{m+2}$
(standard basis elements).
Therefore
\be\label{eq:pm1}
 \begin{aligned}f([v_1,&\cdots,v_{m+2n}])=f([g_0(v_1),\cdots,g_0(v_{m+2n})])\\
=&f([\mu\inv v'_1, \cdots, \mu\inv v_m',\mu\inv e_{m+1},\mu\inv e_{m+2},\cdots,\mu\inv v_{m+2n}'])\\
&\text{ where }v_i'=\mu g_0 (v_i),\;i\neq m+1,m+2\\
=&\mu^{-2d}f([v'_1,v'_2,\dots v_m',e_{m+1}, e_{m+2},\cdots,v_{m+2n}'])\\
=&(-\la)^{d}f([v'_1,v'_2,\dots v_m',e_{m+1}, e_{m+2},\dots,v_{m+2n}'])\\
=&(-(v_{m+1},v_{m+2}))^{d}f([v'_1,v'_2,\dots v_m',e_{m+1}, e_{m+2},v_{m+3}',\dots,v_{m+2n}']).\\
\end{aligned}
\ee

Now for $i\in\{1,\dots,m+2n\}$ and $i\neq m+1, m+2$, there are unique functions $x_i,y_i$ (of $A$) such that
\[
v_i'':=v_i'-e_{m+1}x_i-e_{m+2}y_i\in\langle e_1,\dots,\hat e_{m+1},\hat e_{m+2},\dots
e_{m+2n}   \rangle_{\La}=\langle e_{m+1},e_{m+2}\rangle^{\perp}.
\]
Write $V'$ (resp. $V'_\C$) for the $\La$-span (resp. $\C$-span) of $\{e_i\mid i\neq m+1,m+2\}$.

The $x_i,y_i$ may be thought of as the two entries in position $m+1$ and $m+2$ of
the vectors $v_i'$, with $v_i''$ having those two entries equal to $0$ and the others equal to those of $v_i'$.
Solving the equations $(e_{m+1},v_i'')=(e_{m+2},v_i'')=0$ yields immediately
that $x_i=(e_{m+2},v_i')$ and $y_i=-(e_{m+1},v_i')$.
Now using the definition of $v_i'$ in \eqref{eq:pm1}, we see easily that
\be\label{eq:pm2}
x_i=\frac{-(v_{m+2},v_i)}{(v_{m+1},v_{m+2})}\text{ and }
y_i=\frac{(v_{m+1},v_i)}{(v_{m+1},v_{m+2})}.
\ee

Now let $\bE'$ and $\bS'$ be the analogues for $V'=\langle e_{m+1},e_{m+2}\rangle^\perp$ of $\bE$ and $\bS$.

Let $V^{m+2(n-1)}=\{[v_1',\dots,v_{m}', v_{m+3}',\dots ,v_{m+2n}']\mid v_i'\in \La_{[i]}\}$
 and consider the function $f'\in\cP[V^{m+2(n-1)}]$ defined by
\[
f'([v_1',\dots,v'_{m}, v'_{m+3},\dots ,v_{m+2n}'])=f([v_1',\dots,v_m',e_{m+1},e_{m+2},v_{m+3}',\dots,v_{m+2n}'])
\]
for $[v_1',\dots,v'_{m}, v'_{m+3},\dots ,v_{m+2n}']\in V^{m+2(n-1)}$.

If $g'\in\OSp(V')$, then since $g'(e_j)=e_j$ for $j=m+1,m+2$ (see \eqref{eq:pm2}), we have
$$
\begin{aligned}
f'([g'(v_1'),&\dots,g'(v'_{m}), g'(v'_{m+3})\dots, g'(v_{m+2n}')])\\
=&f([g'(v_1'),\dots,g'(v'_{m}), e_{m+1},e_{m+2},g'(v'_{m+3})\dots, g'(v_{m+2n}')])\\
=&f([g'(v_1'),\dots,g'(v'_{m}), g'(e_{m+1}),g'(e_{m+2}),g'(v'_{m+3})\dots, g'(v_{m+2n}')])\\
=&f([v_1',\dots,v'_{m}, e_{m+1},e_{m+2},v'_{m+3}\dots, v_{m+2n}'])\\
=&f'([v_1',\dots,v'_{m},v'_{m+3}\dots, v_{m+2n}']).\\
\end{aligned}
$$
Thus $f'$ is invariant under $\OSp(V')$.

Now using the relations $v_i'=v_i''+x_ie_{m+1}+y_ie_{m+2}$,
$f'([v_1',\dots,v'_{m}, v'_{m+3},\dots ,v_{m+2n}'])$
may be expanded as a sum of monomials in the $x_i$ and $y_i$.
\be\label{eq:pm3}
f'([v_1',\dots,v'_{m}, v'_{m+3},\dots ,v_{m+2n}'])=
\sum_{I,J}f_{I,J}([v_1'',\dots,v''_{m}, v''_{m+3},\dots ,v_{m+2n}''])x^Iy^J,
\ee
where the sum is over all (relevant, since $x_i,y_i$ may be either
polynomial variables or Grassmann variables, depending on $i$)
monomials $x^Iy^J$ in the $x_i,y_i$. Note that $f'$ is a polynomial
in the entries of the matrix $[v_1',\dots,v'_{m}, v'_{m+3},\dots ,v_{m+2n}']$, which has size
$(m+2n)\times (m+2n-2)$, and whose $(m+1)^{\text st}$ and $(m+2)^{\text nd}$ rows
are respectively $x_i$ and $y_i$. Thus
 $[v_1'',\dots,v''_{m}, v''_{m+3},\dots ,v_{m+2n}'']$
is the submatrix obtained by deleting rows $m+1$ and $m+2$.

Since the $x_i,y_j$ are invariant under $\OSp(V')$
(see \eqref{eq:pm2}),  it follows from the invariance of $f'$ under $\OSp(V')$, as well as the uniqueness
of the expression \eqref{eq:pm3} for any function $f'\in\cP[V^{m+2n-2}]$ (see above),
 that each polynomial $f_{I,J}\in\cP[\bE']$ is invariant under $\OSp(V')$. We may therefore apply to the
$f_{I,J}$ the induction hypothesis that Lemma \ref{lem:key3} holds for $V'$,
 to deduce that there are polynomials $F_{I,J}\in\cP[\bS']$ such that for $A'\in\bE'$,
$f_{I,J}(A')=F_{I,J}(\omega_s(A'))$. That is, $f_{I,J}([v_1'',\dots,v''_{m}, v''_{m+3},\dots ,v_{m+2n}''])$
is a polynomial in the coordinate functions $(v_i'',v_j'')$.

But in terms of the original element $[v_1,\dots,v_{2n}]\in\bE$, an easy calculation shows that
we have, for any pair $i,j\neq m+1,m+2$, integers
$n_{ij},m_{ij}$ such that (recalling that $[x_i]=[y_i]=[i]+\bar1$)
\be\label{eq:pm4}
\begin{aligned}
(v_i'',v_j'')=&\mu^2(v_i,v_j)+n_{ij}x_iy_j+m_{ij}y_ix_j\\
=&-\frac{(v_i,v_j)(v_{m+1},v_{m+2})+n_{ij}(v_{m+2},v_i)(v_{m+1},v_j)+
m_{ij}(v_{m+1},v_i)(v_{m+2},v_j)}{(v_{m+1},v_{m+2})^2}.\\
\end{aligned}
\ee

Substituting \eqref{eq:pm4} into \eqref{eq:pm3}, and then into \eqref{eq:pm1}, we see that for
an element $A=[v_1, v_2,\dots,v_{m+2n}]\in U_1\subset\bE$
(so that $(v_{m+1},v_{m+2})\in\La^*$), we have
\be\label{eq:pm5}
f([v_1,\dots,v_{m+2n}])=u_{m+1,m+2}^{-r}F(\omega_s(A)),
\ee
where $F$ is a polynomial in the coordinate functions $u_{ij}$ of $\bS$ and $r\in\N$.

This proves part (i) of Lemma \ref{lem:ind}, since evidently $U_1\supseteq U$.

We now turn to the {\bf proof of Lemma \ref{lem:key3} (ii)}.
We continue to assume that $f\in\cP[\bE]$ is homogeneous of degree $2d$ and satisfies $f([g(v_1),\dots,g(v_{m+2n})])
=f([v_1,\dots,v_{m+2n}])$ for all $g\in\OSp(V)$.
We are given that $m>0$.

Let $U_2=\{A:=[v_1,\dots,v_{m+2n}]\in\bE\mid (v_1,v_1)\in\La^*\}$.
If $A=[v_1,\dots,v_{m+2n}]\in U_2$, then $(v_1,v_1)=\la\in\La^*$, and there is an element
$\mu\in\La^*$ such that $\la=\mu^{-2}$.
So $(\mu v_1,\mu v_1)=1$ and there is an element $g_0\in\OSp(V)$ such that $g_0(\mu v_1)=e_1$
 (the standard basis element).

As in the proof of part (i), we then have
\be\label{eq:f1}
\begin{aligned}
f(A)=&f([v_1,\dots,v_{m+2n}])=f([g_0(v_1),\dots,g_0(v_{m+2n})])\\=&
f([\mu\inv e_1,\mu\inv v_2',\dots,\mu\inv v'_{m+2n}])
=\mu^{-2d}f([e_1,v_1',\dots,v'_{m+2n}])\\
=&(v_1,v_1)^df([e_1,v_1',\dots,v'_{m+2n}]),\\
\end{aligned}
\ee
where $v_i'=\mu g_0(v_i)$ for $i=2,\dots,m+2n$.

Now write $V^{m-1+2n}$ for the set of $[v_2',\dots,v'_m,v'_{m+1},\dots,v'_{m+2n}]$ such that
$v_i'\in V_{\bar0}$ for $2\leq i\leq m$,
and $v'_i\in V_{\bar1}$ for $m+1\leq i\leq m+2n$, and let $f'\in\cP[V^{m-1+2n}]$ be defined by
\be\label{eq:f2}
f'([v_2',\dots,v_{m+2n}'])=f([e_1,v_2',\dots,v_{m+2n}']).
\ee

Let $V'_\C=\sum_{i=2}^{m+2n} \C e_i$ and write $V'=V'_\C\ot_\C\La$. Denote by $\bE'$ and $\bS'$
 the analogues for
$V'$ of $\bE$ and $\bS$ for $V$. Then for $g\in\OSp(V')=\{g\in\OSp(V)\mid g(e_1)=e_1\}$,
we have $f'([g(v_2'),\dots,g(v_{m+2n}')])=f([g(e_1),g(v_2'),\dots,g(v_{m+2n}')]=
f([e_1,v_2',\dots,v_{m+2n}'])=f'([v_2',\dots,v_{m+2n}'])$.

Thus $f'$ is invariant under $\OSp(V')$. We next use $f'$ to define $\OSp(V')$-invariant functions on $\bE'$.

There exist unique functions $z_i$ of $A\in U\subseteq\bE$ such that for $i=2,\dots, m+2n$,
\be\label{eq:f3}
v_i'':=v_i'-e_1z_i\in V'.
\ee
In fact the element $v_i''$ in \eqref{eq:f3} lies in $V'$ $\iff$ $(e_1,v_i'')=0$, and a simple calculation
shows that this is true if and only if, for $i=2,\dots,m+2n$,
\be\label{eq:f4}
z_i=\frac{(v_1,v_i)}{(v_1,v_1)}.
\ee
Note that as a function on $U_2$, $z_i$ is invariant under $\OSp(V')$.

Now using the relations $v_i'=v_i''+z_ie_1$, of \eqref{eq:f3}, the function $f'$ of \eqref{eq:f2} may be expanded
as a polynomial in the $z_i$, with coefficients polynomials in the $v_i''$.
\be\label{eq:f5}
f'([v_2',\dots,v_{m+2n}'])=\sum_I f_I([v_2'',\dots,v_{m+2n}''])z^I,
\ee
where the sum is over all monomials $z^I$ in the $z_i$, and for each $I$, $f_I\in\cP[\bE']$.

Note that $f'([v_2',\dots,v_{m+2n}'])$ is a (super)-polynomial in the entries of the matrix $[v_2',\dots,v_{m+2n}']$,
of which the first row is $[z_2,\dots,z_{m+2n}]$, and $[v_2'',\dots,v_{m+2n}'']$ is the lower
$(m-1+2n)\times (m-1+2n)$ submatrix. Note also that as polynomial function, $z_i$ is a `polynomial variable'
for $2\leq i\leq m$, and a `Grassmann variable' for $i>m$. Thus each polynomial function on $V^{m-1+2n}$
has a unique expression of the form \eqref{eq:f5}. The invariance of $f'$ under $\OSp(V')$ therefore shows
(since the $z_i$ are invariant under $\OSp(V')$) that each polynomial $f_I$ is invariant under $\OSp(V')$.

We may therefore apply the induction hypothesis to $V'$ to conclude that for each $I$, there is a polynomial
$F_I\in\cP[\bS']$ such that $f_I([v_2',\dots,v_{m+2n}''])=F_I((v_i'',v_j''))$.
Moreover a simple calculation shows that in terms of the original variables $v_i$,
\be\label{eq:f6}
(v_i'',v_j'')=\frac{(v_1,v_1)(v_i,v_j)-(v_1,v_i)(v_1,v_j)}{(v_1,v_1)^2}.
\ee

Note that since $[v_i'']=[v_i']=[z_i]=(v_1,v_i)=[i]$, it follows that $(v_i'',v_j'')=(-1)^{[i][j]}(v_j'',v_i'')$.

Substituting \eqref{eq:f6} and \eqref{eq:f4} into \eqref{eq:f5}, and then into \eqref{eq:f1}, we see that for
$A\in U_2\subset\bE$, we have a regular function $F_1\in\cP[\bS]$ such that for $A\in U_2$,
\be\label{eq:f7}
f(A)=f([v_1,\dots, v_{m+2n}])=(v_1,v_1)^{-r_1}F_1(\omega_s(A)),
\ee
for some $r_1\in\Z$.

This completes the proof of Lemma \ref{lem:ind}.
\end{proof}

\begin{proof}[Completion of the proof of Lemma \ref{lem:key3}]
The proof is by induction on $\sdim(V)=(m|2n)$. If $m+n=1$ the result is true by Lemma \ref{lem:baby}.
We therefore take $m+n$ to be greater than $1$, and assume the result for all superspaces of smaller dimension.
Let $U$ be the (dense) subset defined in Lemma \ref{lem:uwdense}.
We consider three cases.

{\it Case 1:} If $m=0$, then $n\geq 2$. By Lemma \ref{lem:ind}(i), there is a subset $U_1\supset U$ of $\bE$
such that $f([v_1,\dots,v_{2n}]=(u_{12}^{-r_1}F)(\omega_s([v_1,\dots,v_{2n}])$ for $[v_1,\dots,v_{2n}]\in U_1$
and $F$ regular on $\bS$. But the same argument shows that we also have a subset $U_2$ with
$U\subseteq U_2\subseteq\bE$,
such that $f([v_1,\dots,v_{2n}]=(u_{34}^{-r_2}F_2)(\omega_s([v_1,\dots,v_{2n}])$ for $[v_1,\dots,v_{2n}]\in U_2$
and $F_2$ regular on $\bS$. It follows that
\be\label{eq:c1}
u_{12}^{r_1}F_2\circ\omega_s=u_{34}^{r_2}F\circ\omega_s
\ee
on the subset $U_1\cap U_2$ of $\bE$. But $U_1\cap U_2\supseteq U$, and $U$ is dense
in $\bE$. Thus the regular functions $u_{12}^{r_1}F_2\circ\omega_s$ and $u_{34}^{r_2}F\circ\omega_s$
coincide on the dense subset $U$ of $\bE$. So
they coincide on $\bE$ whence the regular functions $u_{12}^{r_1}F_2$ and $u_{34}^{r_2}F$ on $\bS$
coincide on $\omega_s(\bE)$ which by
Lemma \ref{lem:omegadense} is dense in $\bS$. Since $u_{12}$ is an even coordinate
function on $\bS$, this shows that
if $r_1>0$, $u_{12}^{r_1}$ divides $F$ in $\cP[\bS]$. It follows that $u_{12}^{-r_1}F$ is regular on $\bS$,
and the regular functions $f$ and $u_{12}^{-r_1}F\circ\omega_s$ agree on the dense subset $U_1\subset\bE$.
Hence they agree on $\bE$, which
completes the proof in this case.

{\it Case 2:} If $n=0$ then $m>1$ and we argue as above, using Lemma \ref{lem:ind}(ii), obtaining the statement that
\be\label{eq:c2}
u_{11}^{r_1}F_2\circ\omega_s=u_{22}^{r_2}F\circ\omega_s
\ee
on a subset $U_1\cap U_2\supseteq U$ of $\bE$,
in analogy with Case 1. The rest of the argument is the same.

{\it Case 3:} If both $m,n\geq 1$, then again we use Lemma \ref{lem:ind} (i) and (ii) to obtain an equation
\be\label{eq:c3}
u_{11}^{r_1}F_2\circ\omega_s=u_{m+1,m+2}^{r_2}F\circ\omega_s
\ee
on a subset $U_1\cap U_2\supseteq U$, in analogy with \eqref{eq:c1}. The rest of the argument is again the same.

This completes the proof of Lemma \ref{lem:key3}, and hence that of Lemma \ref{lem:key}.
\end{proof}
\section{The super Pfaffian and invariants of $\osp(V_\C)$}\label{sect:pfaffian}

If the Harish-Chandra pair $(G_0,\osp(V_\C))$ is replaced by simply $\osp(V_\C)$
it is in general no longer true that the Brauer algebra maps surjectively onto the endomorphism algebra.
The situation is not dissimilar to the comparison between the invariants of $\fso(m)$ and those of ${\rm O}(m)$
\cite{P2}.
Here we make some comments about the additional invariants required,
which all come from the super Pfaffian discovered by Sergeev \cite{S0}.

\subsection{The super Pfaffian}
%
%

Let $V_\C$ be a complex superspace with $\sdim(V_\C)=(m|2n)$. Given a non-degenerate
graded-symmetric even form $(-,-)$ on $V_\C$, we denote the corresponding orthosymplectic Lie
superalgebra $\osp(V_\C)$ by $\fg$ throughout this section.

\subsubsection{The invariant integral}
Let us recall from \cite{SZ01, SZ05} some elementary facts about the left $\fg$-invariant (Hopf) integral on
the finite dual Hopf superalgebra $\U^0(\fg)$ of the universal enveloping superalgebra of
$\fg$.  First, recall that $\U^0(\fg)$ is the Hopf subalgebra of the dual superalgebra which is generated by
the coordinate functions of the finite dimensional representations of $\fg$ (or $\U(\fg)$). That is, $\U^0(\fg)$
consists of the elements of the dual space of $\U(\fg)$ whose left and right translations under $\fg$ span
a finite dimensional space.

Similarly, denote by $\U(\fg_{\bar0})^0$ the finite dual Hopf algebra of $\U(\fg_{\bar0})$.
Let $p: \U(\fg)^0\longrightarrow \U(\fg_{\bar0})^0$ be the restriction induced by the natural
Hopf superalgebra embedding $\U(\fg_{\bar0})\longrightarrow\U(\fg)$.  There exists a unique (left and right)
$\fg_{\bar0}$-invariant integral $\int_0: \U(\fg_{\bar0})^0\longrightarrow \C$ such that $\int_0 1=1$, so
that we have the composition $\int_0 p=\int_0\circ p: \U(\fg)^0\longrightarrow\C$,
which can be regarded as an element of $(\U(\fg)^0)^*$.
Consider  $Z=\U(\fg)/\U(\fg)\fg_{\bar0}$ as a $\fg$-module, and let $Z^\fg$ be its invariant submodule.
It was shown in \cite{SZ05} (and in \cite{SZ01} for various cases) that  $\dim Z^\fg = 1$.

Fix a generator $z+\U(\fg)\fg_{\bar0}$ of $Z^\fg$ and let
$\nu: \U(\fg)\longrightarrow(\U(\fg)^0)^*$ be the superalgebra embedding
defined for any $x\in\U(\fg)$ by $\nu(x)(f)= (-1)^{[f][x]}\langle f, x\rangle$ for all $f\in\U(\fg)^0$.
Now the comultiplication in $\U(\fg)^0$ gives a multiplication in $(\U(\fg)^0)^*$, defined by
$\psi.\phi(f)=\sum\psi(f_1)\phi(f_2)$, for $\psi,\phi\in(\U(\fg)^0)^*$ and $f\in\U(\fg)^0$, 
where $\Delta(f)=\sum f_1\ot f_2$. Therefore we may define  the following element
of  $(\U(\fg)^0)^*$.
\begin{eqnarray} \label{eq:int}
\int:= \nu(z).\int_0 p.
\end{eqnarray}
Then $\int$ is a nontrivial left $\fg$-invariant integral on $\U(\fg)^0$, in the sense that for any $u\in \U(\fg)$,
\[
\nu(u).\int=1_{\U(\fg)^0}(u)\int.
\]
Moreover $\int$ is independent of the representative $z$ chosen for
$z+\U(\fg)\fg_{\bar0}$, and the integral is also right invariant and is unique up to a scalar multiple.
Note that $\epsilon:=1_{\U(\fg)^0}: \U(\fg)\longrightarrow \C$ is the co-unit of $\U(\fg)$, which is an algebra map
satisfying $\epsilon(1)=1$ and $\epsilon(X)=0$ for all $X\in\fg$.

Any locally finite $\U(\fg)$-module $M$ may be regarded as a right $\U(\fg)^0$-comodule with the structure map
$\omega: M\longrightarrow M\otimes \U(\fg)^0$, $m\mapsto \sum_{(m)} m_{(0)}\otimes f_{(1)}$,
defined by
\[
\sum_{(m)} m_{(0)}  \langle f_{(1)}, x\rangle = (-1)^{[m][x]}x m, \quad \forall x\in U(\fg).
\]
Let $\cI_M:=(\id_M\otimes\int)\omega$. Then one sees easily that $\cI_M(M)\subseteq M^\fg$.
Local finiteness of $M$ implies that it is semi-simple when regarded as a $\U(\fg_{\bar0})$-module.
Thus $M$ decomposes uniquely into $M=M^\perp\oplus M^{\fg_{\bar0}}$ as $\fg_{\bar0}$-module,
and $(\id_M\otimes\int_0 p)\omega$ is the projection of $M$ onto $M^{\fg_{\bar0}}$. Further, unravelling
the definition of $\int$, we see that for any $m=m^\perp+m^0$ with $m^\perp\in M^\perp$ and
$m^0\in M^{\fg_{\bar0}}$, we have  $\cI_M(m) = z m^0$.

\begin{remark} The category of finite dimensional $\fg$-modules is semi-simple
only if $\fg=\osp(1|2n)$. Thus by \cite[Proposition 2]{SZ01},
$\cI_M(M^\fg)=M^\fg$ if $\fg=\osp(1|2n)$ and $\cI_M(M^\fg)=0$ otherwise.
\end{remark}

\subsubsection{The super Pfaffian}

Let $V_\C^m=V_\C\otimes\C^m$, where $\C^m$ is purely even.
Denote by $\cS=S(V_\C^m)$ the supersymmetric algebra (Definition \ref{def:symalg})
 of $V_\C\otimes\C^m$ over $\C$.
The orthosymplectic Lie superalgebra $\fg$ acts on the left factor of  $V_\C^m$, and we  thus have a natural
$\fg$-action on $\cS$.  Note that $\Lambda\otimes\cS$ is isomorphic
to the superalgebra of polynomial functions on
(the dual of) $V^m_{\bar0}=V_{\bar0}\otimes\C^m$ as a $\GL(V)$-module (cf. Proposition
\ref{prop:poly-polyfn}).

To give a concrete description of the super-Pfaffian,
we take $B=(e_1, e_2, \dots, e_{m+2n})$ to be an ordered homogeneous
orthosymplectic basis of $V_\C$, and let $f_1=(1\ 0 \ 0 \dots \ 0)$, $f_2=(0 \ 1\ 0 \dots \ 0)$,
$\dots$, $f_m=(0\ 0 \dots 0 \ 1)$ be the standard basis of $\C^m$.
Then $\cS$ is the tensor product of the polynomial algebra in $x_{i j}=e_i\otimes f_j$ (for $1\le i, j\le m$), and
the Grassmann (exterior) algebra generated by
$\theta_{\mu j}=e_{m+\mu}\otimes f_j$ (for $1\le\mu\le 2n$, $1\le j\le m$).

Now the natural $\gl(V_\C)\times \gl_m$ action on $V_\C^m$ induces an action on $\cS$.
A version of Howe duality states that as $\gl(V_\C)\times \gl_m$-module,
$
\cS= \bigoplus_{\lambda} L_\lambda\otimes W_\lambda,
$
where $L_\lambda$ (resp. $W_\lambda$) denotes the irreducible $\gl(V_\C)$-module (resp.
$\gl_m$-module) with highest weight $\lambda$, and the direct sum is over all partitions
$\lambda$ of depth $\le m$. Here highest weight modules are defined with respect to the
standard Borel subalgebra $\fb$ of $\gl(V_\C)$ spanned by the matrix units $E_{a b}$
(relative to the basis $B$)
with $1\le a\le b\le m+2n$.

If $\cI_\cS$ is the special case for $M=\cS$
of the map $\cI_M$ defined earlier in general, then $\cI_\cS(\cS)\subset \cS^\fg$.
To understand the invariants in $\cI_\cS(\cS)$, we note that $Z$ is isomorphic to $\wedge \fg_{\bar1}$
as a vector space, and hence any generator of $Z^\fg$ contains a nonzero term in the top degree.
Given the generator $z+\U(\fg)\fg_{\bar0}$ of $Z^\fg$, we express $z$ as a sum of PBW
basis elements involving only elements of $\fg_{\bar1}$.
Regarded as an element of $\U(\gl(V_\C))$,
the element $z$ is a sum of weight vectors (with respect to the adjoint action)
of $\gl(V_\C)_{\bar0}\cong \gl_m(\C)\oplus\gl_{2n}(\C)$, and we denote by $z_+$ the component with highest weight.
Then from the embedding $\fg\subset\gl(V_\C)$ we see that $z_+\in
\wedge^{2mn} \gl(V_\C)_{\bar1}$,
where $\gl(V_\C)_{\bar 1}$ is the span of $E_{i, m+\nu}$ with $1\le i\le m$ and $1\le \nu\le 2n$.
After multiplying $z$ by an appropriate nonzero scalar,
we may assume that $z_+$ is equal to the PBW basis element in $\U(\gl(V_\C))$
where the $E_{i, m+\nu}$ occur in a certain order.  Note that
$g z_+= \det(g)^{2n} z_+$ for all $g\in\GL((V_\C)_{\bar0})\times\GL((V_\C)_{\bar1})$.

Denote by $G_0$ the subgroup ${\rm O}((V_\C)_{\bar0})\times{\rm Sp}((V_\C)_{\bar1})$
of the orthosymplectic supergroup. Then $G_0$ acts on $\fg$ by conjugation.
This action induces an action on $Z$.  In particular,
for any $g\in G_0$,
\begin{eqnarray}\label{eq:z-group-inv}
g(z+\U(\fg)\fg_{\bar0})=\det(g)^{2n} (z+\U(\fg)\fg_{\bar0}) = z+\U(\fg)\fg_{\bar0}.
\end{eqnarray}
It follows that $z+\U(\fg)\fg_{\bar0}$ is an invariant of the orthosymplectic supergroup.

Taking the above discussion into account, we now have the following result.
\begin{lemma}\label{lem:construct}
\begin{enumerate}
\item
Let $L_\lambda\subset\cS$ be a simple $\gl(V_\C)$-submodule which is typical, and assume that a nonzero
$\delta_0\in (L_\lambda)^{\fg_{\bar0}}$ is either a $\gl(V_\C)$-highest weight vector or
a $\gl(V_\C)$-lowest weight vector. Then $\delta:=\cI_\cS(\delta_0)$ is a nonzero element of $\cS^\fg$.
\item If for all $g\in {\rm O}((V_\C)_{\bar0})\times{\rm Sp}((V_\C)_{\bar1})$, the $\fg_{\bar0}$-invariant
$\delta_0\in \cS^{\fg_{\bar0}}$ satisfies $g\cdot \delta_0=\det(g) \delta_0$, then $g\cdot \delta=\det(g) \delta$.
\end{enumerate}
\end{lemma}
\begin{proof}
For part (1), we assume that $\delta_0$ is a $\gl(V_\C)$-lowest weight vector of $L_\lambda\in\cS$.
Then typicality of $L_\lambda$ implies that $z^+\delta_0\ne 0$ and hence $z\delta_0\ne 0$.
If $\delta_0$ is a $\gl(V_\C)$-highest weight vector, we proceed similarly by
considering the lowest weight component $z_-$ of $z$. Part (2) follows from \eqref{eq:z-group-inv}.
\end{proof}

Part (2) of the lemma asserts that if $\delta_0\in\cS$ is a $\fg_{\bar0}$-invariant but not an invariant of
${\rm O}((V_\C)_{\bar0})\times{\rm Sp}((V_\C)_{\bar1})$, then $\delta\in \cS^\fg$ will
not be an invariant of the orthosymplectic supergroup.
We now turn to the construction of such a $\delta_0$.

Let $\Delta=\det X$, where $X=(x_{i j})$ is the $m\times m$ matrix with entries $x_{ij}$ as above.
Set $\Pi=\prod_{\mu, j} \theta_{\mu j}$ for some
fixed ordering of the elements $\theta_{\mu j}$; evidently $\Pi$ is in the top degree
component of the Grassmann algebra.
Then $\Delta \Pi$ is a $\gl(V_\C)$-lowest weight vector (i.e. killed by
the strictly lower triangular part of $\gl(V_\C)$), whose weight is $(1^m\mid m^{2n})$.
Here we write a $\gl(V_\C)$ weight as
$\lambda=(\lambda_{\bar0} \mid \lambda_{\bar1})$ with $\lambda_{\bar0}$ (resp. $\lambda_{\bar1}$)
being the corresponding $\gl((V_\C)_{\bar0})$ (resp. $\gl((V_\C)_{\bar1})$) weight. The simple
$\gl(V_\C)$-module with this lowest weight vector is $L_{((2n+1)^m \mid 0)}$,
which is typical. As one can immediately see from its weight,
$\Delta \Pi$ spans a simple $\gl(V_\C)_{\bar0}$-module in $\cS$.
It  is even and homogeneous of degree $m(2n+1)$.
\begin{definition}\label{def:Pfaffian}
Let $\Omega:= z(\Delta \Pi)$ and call it the super Pfaffian.
\end{definition}
\begin{remark}
We believe that $\Omega$ agrees with the super Pfaffian defined by Sergeev \cite{S1} up to a sign,
hence the terminology.
\end{remark}

By Lemma \ref{lem:construct},  $\Omega$  is a nonzero $\osp(V_\C)$-invariant, which is not an invariant of the
corresponding orthosymplectic supergroup. However $\Omega^2$ is $\OSp(V)$-invariant.
Note that $\Omega$ is even and homogeneous and has the same degree as $\Delta \Pi$.
A simple computation shows that the leading term $z_+(\Delta \Pi)$ of $\Omega$
is equal to $\Delta^{1+2n}$ up to a sign.  Thus the leading term of $\Omega^2$ is equal to
$\det(X^tX)^{1+2n}$.

\begin{remark}\label{rem:sergeev}
Sergeev's \cite[Theorem 1.3]{S1} states that the subalgebra of
$\osp(V_\C)$-invariants in $S(V_\C\otimes \C^{p|q})$
is generated by the quadratics of Corollary \ref{cor:fft-osp-poly},
as well as his super Pfaffian if $p\ge m$.
We will give an independent proof of our version of this result elsewhere.
\end{remark}

\subsection{Comments on endomorphism algebras of $\osp(V_\C)$-modules}

The supersymmetric algebra $\cS=S(V_\C^m)$ is isomorphic to $S(V_\C)^{\otimes m}$, and the super Pfaffian
belongs to $\left(S^{2n+1}(V_\C)\right)^{\otimes m}$, where $S^{2n+1}(V_\C)$ is the
degree $2n+1$ homogeneous component of $S(V_\C)$.
Now $\left(S^{2n+1}(V_\C)\right)^{\otimes m} \hookrightarrow V_\C^{\otimes m(2n+1)}$ as module
for $\fg=\osp(V_\C)$.

Denote by $\tilde\Omega$ the image of the super Pfaffian in $V_\C^{\otimes m(2n+1)}$. Then
$\tilde\Omega\in \left(V_\C^{\otimes m(2n+1)}\right)^\fg$ and
$\tilde\Omega\otimes\tilde\Omega\in \left(V_\C^{\otimes 2m(2n+1)}\right)^G$, where
$G=\OSp(V)$ is interpreted as the Harish-Chandra
super pair $({\rm O}((V_\C)_{\bar0})\times {\rm Sp}((V_\C)_{\bar1}), \osp(V_\C))$.

When $m$ is even, we write $r_c= \frac{m}{2}(2n+1)$. Let $\kappa\in\End_\fg\left(V_\C^{\otimes r_c}\right)$
be the image of $\tilde\Omega$ under the isomorphism
$\left(V_\C^{\otimes 2r_c}\right)^\fg\stackrel{\sim}{\longrightarrow}
\End_\fg(V_\C^{\otimes r_c})$.
Then $\kappa\not\in \End_G\left(V_\C^{\otimes r_c}\right)$.
Thus if $m$ is even and $r\ge r_c$, it follows from Corollary \ref{cor:Brauer} that
\[
\End_\fg\left(V_\C^{\otimes r_c}\right)\supsetneq
F_{r_c}^{r_c}(B_{r_c}(m-2n)_\C),
\]
where $B_{r_c}(m-2n)_\C$ is the complex Brauer algebra of degree $r_c$ with parameter $m-2n$.

It can be shown that $\End_{\osp(V_\C)}\left(V_\C^{\otimes r}\right)$ is generated by
$\kappa$ and $F_r^r(B_r(m-2n)_\C)$ if $r\ge r_c$ and $m$ is even; finally,
$\End_{\osp(V_\C)}\left(V_\C^{\otimes r}\right)$ coincides with $\End_{\OSp(V)}\left(V_\C^{\otimes r}\right)$
if $m$ is odd or $r<r_c$.



\begin{thebibliography}{9999}
\bibitem{ABP} Atiyah, M.; Bott, R.; Patodi, V. K., ``On the heat
    equation and the index theorem'', {\sl Invent. Math.  \bf 19}  (1973), 279--330.

\bibitem{BLR} Benkart, Georgia; Shader, Chanyoung Lee; Ram, Arun,
``Tensor product representations for orthosymplectic Lie superalgebras",
{\sl J. Pure Appl. Algebra \bf 130} (1998), no. 1, 1 -- 48.

\bibitem{BR} A. Berele and A. Regev, ``Hook Young diagrams with applications
to combinatorics and to representations of Lie superalgebras",  {\sl Adv.  Math. \bf 64} (1987), 118-175.

\bibitem{BS}
Jonathan Brundan  and Catharina Stroppel, ``Gradings on walled Brauer algebras and Khovanov's arc algebra",
{\sl Advances Math. \bf 231} (2012), 709--773.

\bibitem{CCF} Carmeli, C.; Caston, L.; Fioresi, R. `` Mathematical foundations of supersymmetry".
EMS Series of Lectures in Mathematics. European Mathematical Society (EMS), Z\"urich, 2011.

\bibitem{CW} Cheng, S.-J.; Wang, W., ``Howe duality for Lie superalgebras",  {\sl Compositio Math. \bf 128}
             (2001), no. 1, 55--94.

\bibitem{CZ} Cheng, S.-J.; Zhang, R. B. ``Howe duality and combinatorial character formula for
         orthosymplectic Lie superalgebras", {\sl Adv. Math. \bf 182} (2004), no. 1, 124--172.

\bibitem{CLZ} Cheng, S.-J.; Lam, N.; Zhang, R. B. ``Character formula for infinite-dimensional
     unitarizable modules of the general linear superalgebra",  {\sl J. Algebra \bf 273} (2004), no. 2, 780--805.


\bibitem{FH} Fulton, William; Harris, Joe, `` Representation theory. A first course'',
Graduate Texts in Mathematics, {\bf 129}. Readings in Mathematics. Springer-Verlag, New York, 1991.

\bibitem{DM} Deligne, Pierre; Morgan, John W. ``Notes on supersymmetry (following Joseph Bernstein)".
    Quantum fields and strings: a course for mathematicians, Vol. {\bf 1}, {\bf 2} (Princeton, NJ, 1996/1997),
    41--97, Amer. Math. Soc., Providence, RI, 1999.

\bibitem{GW} R. Goodman and N.R. Wallach,
    ``Representations and Invariants of the Classical Groups",
    Cambridge University Press, third corrected printing, 2003.

\bibitem{GW1} R. Goodman and N.R. Wallach,
``Symmetry, representations, and invariants'',
 Graduate Texts in Mathematics, {\bf 255}, Springer, Dordrecht, 2009.

\bibitem{GL96} J.J. Graham and G.I. Lehrer, ``Cellular algebras'',
    {\sl Inventiones Math. \bf 123} (1996), 1--34.

\bibitem{GL98} J.~J.~Graham and G.~I.~ Lehrer,
``The representation theory of affine Temperley-Lieb algebras'',
    {\sl  Enseign. Math. (2) \bf 44} (1998), no. 3-4, 173--218.

\bibitem{GL03} J.J.~Graham and G.I. Lehrer, ``Diagram algebras,
    Hecke algebras and decomposition numbers at roots of unity'',  {\sl
    Ann. Sci. \'Ecole Norm. Sup. \bf 36}  (2003), 479--524.

\bibitem{GL04} J.J.~Graham and G.I. Lehrer,
    ``Cellular algebras and diagram algebras in representation theory'',
    {\sl Representation theory of algebraic groups and quantum groups,
    Adv. Stud. Pure Math. {\bf 40}}, Math. Soc. Japan, Tokyo, (2004), 141--173.

\bibitem{GZ} Gould, M. D.; Zhang, R. B. ``Classification of all star irreps of $gl(m|n)$",
{\sl J. Math. Phys. \bf 31} (1990), no. 11, 2552--2559.

\bibitem{H} Roger Howe, ``Remarks on classical invariant theory",
{\sl Trans. Amer. Math. Soc. \bf 313} (1989), no. 2, 539--570.

\bibitem{K} V. Kac, ``Lie superalgebras",  {\sl Adv. Math. \bf 26} (1977), no. 1, 8--96.

\bibitem{Ko} B. Kostant, ``Graded manifolds, graded Lie theory, and prequantization".  In
{\em Differential geometrical methods in mathematical physics (Proc. Sympos., Univ. Bonn, Bonn, 1975)},
pp. 177--306. Lecture Notes in Math., Vol. {\bf 570}, Springer, Berlin, 1977.

\bibitem{LZ1} G.~I.~Lehrer and R.~B.~Zhang, ``Strongly
multiplicity free modules for Lie algebras and quantum groups'',
{\sl J. of Alg. \bf 306} (2006), 138--174.

\bibitem{LZ2} G.~I.~Lehrer and R.~B.~Zhang,
``A Temperley-Lieb analogue for the BMV algebra",
in {\em Representation theory of algebraic groups and quantum groups}, 155--190,
Progr. Math., {\bf 284}, Birkh\"auser/Springer, New York, 2010.

\bibitem{LZ3} G.I. Lehrer and R.B. Zhang,
``On endomorphisms of quantum tensor space", {\sl Lett. Math. Phys. \bf 86} (2008),  209--227.

\bibitem{LZ4} G. I. Lehrer and R. B. Zhang,
    ``The second fundamental theorem of invariant theory for the orthogonal group",
    {\sl Annals of Math. \bf 176} (2012), 2031--2054.

\bibitem{LZ5} G. I. Lehrer and R. B. Zhang, ``The Brauer Category and Invariant Theory", 	
    arXiv:1207.5889 [math.GR].

\bibitem{LZZ} G.I. Lehrer, Hechun Zhang and R.B. Zhang, ``A quantum analogue of
    the first fundamental theorem of invariant theory",  {\sl Commun. Math. Phys.
    \bf 301} (2011), 131--174.

\bibitem{M} Yuri I. Manin,  ``Gauge field theory and complex geometry".
Grundlehren der Mathematischen Wissenschaften, {\bf 289}. Springer-Verlag, Berlin, 1988.

\bibitem{P2} Procesi, Claudio,  ``Lie groups. An approach through invariants and
representations'', Universitext. Springer, New York, 2007. xxiv+596 pp.

\bibitem{RS} Hebing Rui and Yucai Su, ``Affine walled Brauer algebras'', arXiv:1305.0450.

\bibitem{SS} Salam, Abdus; Strathdee, J. ``Super-gauge transformations",
 {\sl Nuclear Phys. \bf B76} (1974), 477--482.

\bibitem{Sch1} Scheunert, M. ``Graded tensor calculus",
{\sl J. Math. Phys. \bf 24} (1983), no. 11, 2658--2670.

\bibitem{Sch2} Scheunert, M.
``Casimir elements of $\varepsilon$ Lie algebras",
{\sl J. Math. Phys. \bf 24} (1983), no. 11, 2671--2680.

\bibitem{Sch3} Scheunert, M.
``Eigenvalues of Casimir operators for the general linear,
the special linear and the orthosymplectic Lie superalgebras",
{\sl J. Math. Phys. \bf 24} (1983), no. 11, 2681--2688.

\bibitem{S} M. Scheunert, ``The theory of Lie superalgebras; an introduction".
Lecture Notes in Math., vol. {\bf 716}, Springer-Verlag, Berlin-Heidelberg-New York, 1979.

\bibitem{SZ01} Scheunert, M.; Zhang, R. B.
``Invariant integration on classical and quantum Lie supergroups", {\sl J. Math. Phys. \bf 42}  (2001), no. 8, 3871--3897.

\bibitem{SZ} Scheunert, M.; Zhang, R. B.
    ``The general linear supergroup and its Hopf superalgebra of regular functions",
    {\sl J. Algebra \bf 254} (2002), no. 1, 44--83.

\bibitem{SZ05} Scheunert, M.; Zhang, R. B.
``Integration on Lie supergroups: a Hopf algebra
approach", {\sl  J. Algebra \bf 292} (2005), 324--342.

\bibitem{S0}A. Sergeev, ``An analogue of the classical theory of invariants for Lie superalgebras",
    (Russian) {\sl Funktsional. Anal. i Prilozhen. \bf 26} (1992), no. 3, 88--90;
    translation in {\sl Funct. Anal. Appl. \bf 26} (1992), no. 3, 223--225.

\bibitem{S1}A. Sergeev,  ``An analog of the classical invariant theory for Lie superalgebras. I",
    {\sl Michigan Math. J. \bf 49} (2001), Issue 1, 113-146.

\bibitem{S2}A. Sergeev,  ``An analog of the classical invariant theory for Lie superalgebras. II",
    {\sl Michigan Math. J. \bf 49} (2001), Issue 1, 147-168.

\bibitem{V} Varadarajan, V. S. ``Supersymmetry for mathematicians: an introduction".
    Courant Lecture Notes in Mathematics, {\bf 11}. New York University, Courant Institute of Mathematical Sciences,
    New York; American Mathematical Society, Providence, RI, 2004.

\bibitem{DeW} De Witt,  Bryce,  ``Supermanifolds". Second edition.
Cambridge Monographs on Mathematical Physics.
    Cambridge University Press, Cambridge, 1992.

\bibitem{WZ} Wu, Yuezhu; Zhang, R. B.
``Unitary highest weight representations of quantum general linear superalgebra",
        {\sl J. Algebra \bf 321} (2009), no. 11, 3568--3593.

\bibitem{Z1} Zhang, R. B. ``Finite-dimensional irreducible representations of
the quantum supergroup $U_q(gl(m/n))$",
            {\sl J. Math. Phys. \bf 34} (1993), no. 3, 1236--1254.

\bibitem{Z2} Zhang, R. B. ``Structure and representations of the quantum general linear supergroup",
        {\sl Comm. Math. Phys. \bf 195} (1998), no. 3, 525--547.

\end{thebibliography}
\end{document}